\documentclass[a4paper,reqno]{amsart}

\usepackage{amssymb}
\usepackage{latexsym}
\usepackage{amsmath}
\usepackage{amsthm}
\usepackage{euscript}
\usepackage{bbm}
\usepackage{tikz}
\usepackage{changes}

\usepackage{todonotes}
\usepackage{hyperref}

\newcommand{\blue}[1]{{\color{black} #1}}

   \def\cE{{\mathcal E}}   
      
      \def\cL{{\mathcal L}}
      
\def\cP{{\mathcal P}}   \def\cQ{{\mathcal Q}}   
      
\def\cV{{\mathcal V}}

\def\cal H{{\mathcal H}}

\def\R{\mathbb{R}}
\def\C{\mathbb{C}}

\def\N{\mathbb{N}}

\def\ran{{\text{\rm ran\,}}}
\def\dom{{\text{\rm dom\,}}}

\def\phi{\varphi}

\DeclareMathOperator{\Real}{Re}

\DeclareMathOperator{\spann}{span}

\renewcommand{\theta}{\vartheta}

\newcommand{\form}{h}
\newcommand{\from}{\colon}

\newtheorem{theorem}{Theorem}[section]
\newtheorem*{thm*}{Theorem}
\newtheorem{proposition}[theorem]{Proposition}
\newtheorem{corollary}[theorem]{Corollary}
\newtheorem{lemma}[theorem]{Lemma}

\newtheorem{openproblem}[theorem]{Open Problem}

\theoremstyle{definition}
\newtheorem{definition}[theorem]{Definition}
\newtheorem{example}[theorem]{Example}
\newtheorem{hypothesis}[theorem]{Hypothesis}

\newtheorem{remark}[theorem]{Remark}

\numberwithin{equation}{section}

\title[Spectral theory for Schr\"odinger operators on metric graphs]{Spectral Theory for Schr\"odinger operators on compact metric graphs with $\delta$ and $\delta'$ couplings: a survey}

\author[J.~Rohleder]{Jonathan Rohleder}
\address[J.\ Rohleder]{Matematiska institutionen\\ Stockholms Universitet \\
106 91 Stockholm \\
Sweden}
\email{jonathan.rohleder@math.su.se}

\author[C.~Seifert]{Christian Seifert}
\address[C.\ Seifert]{Technische Universit\"at Hamburg \\ Institut f\"ur Mathematik \\
Am Schwar\-zen\-berg-Campus~3 \\
Geb\"aude E \\
21073 Hamburg \\
Germany}
\email{christian.seifert@tuhh.de}

\date{\today}

\keywords{metric graph, quantum graph, Schr\"odinger operator, eigenvalues, surgery principles}

\subjclass[2020]{Primary: 47D08, 34L40, 35P05; Secondary: 81Q10}

\thanks{This article is based upon work from COST Action 18232 MAT-DYN-NET, supported by COST (European Cooperation in Science and Technology), \url{www.cost.eu}. J.R.\ acknowledges financial support by the Swedish Research Council (VR), grant no.\ 2022-03342. The authors are grateful to James B.\ Kennedy for some useful remarks.}

\begin{document}

\begin{abstract}
Spectral properties of Schr\"odinger operators on compact metric graphs are studied and special emphasis is put on differences in the spectral behavior between different classes of vertex conditions. We survey recent results especially for $\delta$ and $\delta'$ couplings and demonstrate the spectral properties on many examples. Amongst other things, properties of the ground state eigenvalue and eigenfunction and the spectral behavior under various perturbations of the metric graph or the vertex conditions are considered. 
\end{abstract}

\maketitle

\section{Introduction}

In the last three decades, differential operators on metric graphs, so-called quantum graphs, have been studied extensively. Metric graphs yield effective one-dimensional models for networks of quasi-one-dimensional structures such as nano\-tubes or waveguides, see, e.g., the survey in \cite[Section 7.6]{BK13}. Moreover, they provide easily accessible models due to its one-dimensional nature on the edges, whereas metric graphs exhibit a non-trivial behavior due to the coupling in the vertices.

By now a vast body of literature studying especially self-adjoint Schr\"odinger operators on metric graphs exists, see e.g.\ \cite{BK13,KS1999,kps06,kps08,Kuchment2004,kuc08,Mugnolo2014,P12} for a non-exhaustive list, but also other types of operators such as first order operators \cite{BolteHarrison2003,SchubertSeifertVoigtWaurick2015,W2012}, operators related to diffusion processes on metric graphs \cite{HM13,kkvw09,KMN2022,SeifertVoigt2011}, as well as higher-order operators have been investigated as well, e.g.\ in \cite{BGM2021,BE22,GM2020,GM2020a,KM21,MNS2018,Seifert2018}. Furthermore, descriptions of certain non-self-adjoint operators on metric graphs and spectral estimates for them appear in the literature \cite{BLLR18,HKS15}.%\todo{\cite{BLLR18} ergänzt}
Besides these continuum models on metric graphs, difference operators acting on the vertices of discrete graphs have a long history, and there is an intimate relation between operators on metric graphs and operators on discrete graphs, see \cite{B85,C97,EKMN18,KN2021,KN2023+}, as well as \cite{KLW2021} for an extensive overview on operators on discrete graphs.

In this article our focus is on Schrödinger operators
\begin{align*}
 - \frac{\textup{d}^2}{\textup{d} x^2} + q
\end{align*}
acting on the edges of a metric graph, equipped with a real-valued potential $q$ and suitable vertex conditions; cf.\ Section \ref{sec:metric-graphs} for details. We will admit general self-adjoint vertex conditions, but often we focus particularly on $\delta$ and $\delta'$ vertex conditions; the most prominent continuity-Kirchhoff (also called standard or Neumann) vertex conditions are a special case of the first-mentioned class, and their dual counterparts, anti-Kirchhoff conditions, belong to the second class.

Many properties of the non-trivial behavior of Schrödinger operators on metric graphs can be investigated in terms of spectral theory, which has been studied in recent years in various perspectives such as spectral estimates, properties of eigenfunctions, inverse problems or questions of isospectrality, see \cite{AB21,ABB18,BSS2022,BKKM19,BerkolaikoLatushkinSukhtaiev2019,BK2022,BCJ21,EKMN18,HM18,HM20,HKMP21,KL20,KR21,KN19,KN21,KR20,MP23,MR21,P21,PT21} for a few of the most recent developments.%\todo{\cite{BSS2022,BK2022} ergaenzt}

In recent years, the effect of geometric manipulations of the metric graph onto the eigenvalues of Schrödinger operators have received much attention. This is mainly due to the fact that such so-called surgery principles have turned out to be powerful tools for obtaining eigenvalue bounds and prove isoperimetric inequalities \cite{BL17,BKKM17,BKKM19,BCJ21,BHY23,KKMM16,KMN13,KN14,KS19,MP20,MR21,R17,R22}. 

In this survey we review different aspects of the spectral theory for self-adjoint Schrödinger operators on metric graphs. While some of the results presented here are valid for general self-adjoint vertex conditions, in many cases we put emphasis on the comparison between $\delta$ and $\delta'$ vertex conditions. Among other things, we focus on the effect of surgery principles. These may be divided basically into two groups: on the one hand manipulations which change the total length of the graph such as 
\begin{itemize}
 \item increasing the length of one or all edges,
 \item adding a new edge or attaching another graph,
 \item inserting another graph into a vertex,
 \item shrinking some edge lengths to zero.
\end{itemize}
And on the other hand surgical operations which preserve the total length of the graph, for instance
\begin{itemize}
 \item changing the coupling condition or its strength at a vertex,
 \item joining vertices,
 \item unfolding parallel or pendant edges.
\end{itemize}
In addition, we study properties of the ground state eigenvalue and the corresponding eigenfunction, which especially are of some use to the above-named surgery principles. Moreover, we discuss bounds for the lowest eigenvalue of the Laplacian with $\delta$ and $\delta'$ vertex conditions, where some of the above surgery principles can be of use. 

This article has mostly the character of a survey. Many of the results come from recent literature in this or a similar form, though we in some cases provide slight generalizations and modified proofs. However, we complement these results by some modest additions such as Theorem \ref{thm:insertDelta'} or Theorem \ref{thm:laaangweilig} and a bunch of illuminating examples and counterexamples. In addition, we formulate some open questions.

Let us outline the content of this article. In Section \ref{sec:metric-graphs} we describe the setup of metric graphs and Schr\"odinger operators on them. Further, we collect the basic facts on spectral theory we need. In Section \ref{sec:groundState} we focus on the ground state eigenvalue and eigenfunction. The subsequent sections are devoted to surgery principles, and they are divided into Section \ref{sec:surgery-principles_edges} on graph modifications which may change the total length and Section \ref{sec:surgery-principles_vertices} on those which preserve it. In Section \ref{sec:Hadamard} we provide Hadamard-type variational formulas which describe the variation of spectral data (we focus on simple eigenvalues here) under variation of the model data; specifically, we consider variation of the coupling conditions. In the final Section \ref{sec:bounds} we discuss, as an application, some bounds for the lowest eigenvalue for $\delta$ and $\delta'$ vertex conditions.

\section{Schr\"odinger operators on metric graphs}
\label{sec:metric-graphs}

\subsection{Metric graphs}

Let $\Gamma$ be a metric graph constituted by a finite set of vertices $\cV := \cV (\Gamma)$, a finite set of edges $\cE := \cE (\Gamma)$ and a length function $L \colon \cE \to (0, \infty)$. Upon parametrizing each edge $e\in \cE$ by the interval $[0, L(e)] \subseteq \R$ and considering $\Gamma$ as the one-dimensional simplicial complex consisting of the intervals $[0,L(e)]$, $e\in \cE$, with corresponding coupling at the vertices, $\Gamma$ becomes a compact metric space, and we will refer to such $\Gamma$ as a compact metric graph. For each vertex $v\in \cV$ we denote by $\cE_v\subseteq \cE$ the set of edges incident to $v$ and by $\deg (v)$ the degree of $v$. We write $v = o (e)$ or $v = t (e)$ and say that the edge $e$ originates from or terminates at the vertex $v$, respectively, if $v$ corresponds to the endpoint $0$ respectively $L (e)$ according to our parametrization. 
% Moreover, we set
% \begin{align*}
%  \sigma_e (v) = \begin{cases} 1, & v = o (e), \\ -1, & v = t (e). \end{cases}
% \end{align*}
We will denote by 
\begin{align*}
 L^2 (\Gamma):=\bigoplus_{e\in \cE} L^2(0,L(e))
\end{align*}
the usual $L^2$ space on $\Gamma$, and by 
\begin{align*}
 \widetilde{H}^k(\Gamma):=\bigoplus_{e\in \cE} H^k(0,L(e))
\end{align*}
the $L^2$-based Sobolev space of oder $k\in\N$. For $f\in L^2(\Gamma)$ and $e\in \cE$ we write $f_e$ for the restriction of $f$ to the edge $e\in \cE$ (i.e.\ the interval $(0,L(e))$ corresponding to $e$).

\subsection{Schr\"odinger operators on metric graphs}

In this paper we focus on self-adjoint Schr\"odinger operators in $L^2 (\Gamma)$ subject to self-adjoint coupling conditions on the vertices; that is, we consider Schr\"odinger operators in $L^2 (\Gamma)$ acting as 
\begin{align}\label{eq:Schroedinger}
 (\cL f)_e := - f_e'' + q_e f_e, \quad e \in \cE,
\end{align}
with a real-valued potential $q=(q_e)_{e\in \cE}$, where we assume that $q \in L^\infty(\Gamma)$. This assumption is made for reasons of simplicity; many of the results discussed in this survey extend naturally to larger classes of potentials such as $L^1 (\Gamma)$.

Next we specify vertex conditions with which $\cL$ will constitute a self-adjoint operator. For $v \in \cV$ let $\{e_1, \dots, e_l\}:=\{e\in \cE_v:\; v = o (e)\}$ and $\{e_{l + 1}, \dots e_{m} \}:=\{e\in \cE_v:\; v = t (e)\}$ be enumerations of the sets of edges originating from and terminating at $v$, respectively. Note that, in the special case that an edge $e$ is a loop, i.e.\ $o (e) = t (e)$, $e$ will belong to both sets. For $f \in \widetilde{H}^1 (\Gamma)$ we write
\begin{align*}%\label{eq:vectors}
 F (v) := \begin{pmatrix} f_{e_1} (0) \\ \vdots \\ f_{e_l} (0) \\ f_{e_{l + 1}} (L (e_{l + 1})) \\ \vdots \\ f_{e_{m}} (L (e_{m})) \end{pmatrix}
\end{align*}
for the collection of boundary values of the restriction of $f$ to the adjacent edges at $v$. Moreover, for $f\in \widetilde{H}^2(\Gamma)$ we write
\begin{align*}F' (v) := \begin{pmatrix} f'_{e_1} (0) \\ \vdots \\ f'_{e_l} (0) \\ - f'_{e_{l + 1}} (L (e_{l + 1})) \\ \vdots \\ - f'_{e_{m}} (L (e_{m})) \end{pmatrix},
\end{align*}
for the collection of derivatives of $f$ at $v$ in the direction pointing out of $v$ into the edges. The following description of all self-adjoint incarnations of  $\cL$ in $L^2 (\Gamma)$ with local coupling conditions is standard.

\begin{proposition}[{e.g.\ \cite[Theorem~1.4.4]{BK13}}]
\label{prop:SAconditions}
Let $\Gamma$ be a compact metric graph, $q \in L^\infty (\Gamma)$ be real-valued and $\cL$ be the Schr\"odinger differential expression in~\eqref{eq:Schroedinger}. For each vertex $v \in \cV$ let $P_{v, \rm D}, P_{v, \rm N}$ and $P_{v, \rm R}$ be orthogonal projections in $\C^{\deg (v)}$ with mutually orthogonal ranges such that $P_{v, \rm D} + P_{v, \rm N} + P_{v, \rm R} = I$ and let $\Lambda_v$ be a self-adjoint, invertible operator in $\ran P_{v, \rm R}$.%\marginpar{\tiny JR: Wie soll man das eigentlich lesen? Ist gemeint, $\Lambda_v$ is eine hermitesche Matrix in $\C^d$ mit $\ran P_{v, \rm R} \cap \ker \Lambda_v = \{0\}$? Formal sind ja Orthogonalprojektionen Operatoren von einem Raum in denselben.CS: $\ran P_R$ ist ein (abgeschlossener) Unterraum von $\C^d$. Dann $\Lambda_v : \ran P_R\to \ran P_R$ s.a. und invertierbar. Dabei ist es unerheblich, ob man die Einbettung von $\ran P_R$ (als Unterraum) in $\C^d$ nutzt und damit $\Lambda_v:\C^d\to \C^d$ liest; in diesem Fall gilt genau das oben.} 
Then the operator $H$ in $L^2 (\Gamma)$ given by
\begin{align*}
 H f & := \cL f, \\
 \dom H & := \Big\{ f \in \widetilde H^2 (\Gamma) : P_{v, \rm D} F (v) = 0, P_{v, \rm N} F' (v) = 0, \\
 & \qquad \qquad P_{v, \rm R} F' (v) = \Lambda_v P_{v, \rm R} F (v)~\text{for each}~v \in \cV \Big\},
\end{align*}
is self-adjoint (and each self-adjoint realization of $\cL$ in $L^2 (\Gamma)$ subject to local coupling conditions can be written in this form). Furthermore, the closed quadratic form $h$ corresponding to the operator $H$ in the sense of, e.g., \cite[Chapter VI, Theorem 2.1]{Kato} is given by
\begin{align*}
 h (f) & := \int_\Gamma |f'|^2  + \int_\Gamma q |f|^2 + \sum_{v \in \cV} \big\langle \Lambda_v P_{v, \rm R} F (v), P_{v, \rm R} F (v) \big\rangle, \\
 \dom h & := \left\{ f \in \widetilde H^1 (\Gamma) : P_{v, \rm D} F (v) = 0~\text{for each}~v \in \cV \right\}.
\end{align*}
\end{proposition}

By a standard compact embedding argument the spectrum of the Hamiltonian $H$ on the compact metric graph $\Gamma$ is always purely discrete and bounded from below, see, e.g.,~\cite[Corollary 10 and Theorem 18]{Kuchment2004}. We denote by
\begin{align*}
 \lambda_1 (H) \leq \lambda_2 (H) \leq \dots 
\end{align*}
the eigenvalues of $H$ in non-decreasing order and counted with multiplicities. If $q\geq 0$ and $\Lambda_v$ is non-negative for each vertex $v$ then all eigenvalues are non-negative. We will frequently make use of the fact that these eigenvalues may be expressed via the min-max principle
\begin{align*}
  \lambda_k (H) = \min_{\substack{F \subseteq \dom h \\ \dim F = k}} \max_{\substack{f \in F \\ f \neq 0}} \frac{h (f)}{\int_\Gamma |f|^2}
\end{align*}
for $k = 1, 2, \dots$.

We will put special emphasis on so-called $\delta$ and $\delta'$ coupling conditions and their special variants, namely continuity-Kirchhoff and anti-Kirchhoff conditions. To specify those within the framework of Proposition \ref{prop:SAconditions}, for $d\in\N$ we denote by
\begin{align*}
 \cP:=\cP_d = \begin{pmatrix} \frac{1}{d} & \dots & \frac{1}{d} \\ \vdots & & \vdots \\ \frac{1}{d} & \dots & \frac{1}{d} \end{pmatrix} \quad \text{and} \quad \cQ:=\cQ_d = \begin{pmatrix} \frac{d - 1}{d} & - \frac{1}{d} & \dots & - \frac{1}{d} \\ - \frac{1}{d} & \ddots & \ddots & \vdots \\ \vdots & \ddots & \ddots & - \frac{1}{d} \\ - \frac{1}{d} & \dots & - \frac{1}{d} & \frac{d - 1}{d} \end{pmatrix}
\end{align*}
the orthogonal projections onto $\spann \{ (1, 1, \dots, 1)^\top \}$ and onto its orthogonal complement $(\spann \{ (1, 1, \dots, 1)^\top \})^\perp$, respectively.

\begin{definition}\label{def:Cond}
    Let $v\in \cV$ and $d:=\deg(v)>0$. 
    Then the vertex conditions at $v$ are called 
    \begin{enumerate}
     \item[(a)] \emph{continuity-Kirchhoff conditions} if $P_{v,\rm D} = \cQ$, $P_{v,\rm N} = \cP$ and $P_{v,\rm R} = 0$;
     \item[(b)] \emph{$\delta$ coupling conditions} with \emph{strength} $\alpha_v\in\R$ if $P_{v,\rm D} = \cQ$, $P_{v,\rm N} = 0$, $P_{v,\rm R} = \cP$ and $\Lambda_v$ is the multiplication by $\frac{\alpha_v}{d}$;
    \item[(c)] \emph{anti-Kirchhoff conditions} if $P_{v,\rm N} = \cQ$, $P_{v,\rm D} = \cP$ and $P_{v,\rm R} = 0$;
     \item[(d)] \emph{$\delta'$ coupling conditions} with \emph{strength} $\beta_v\in\R\setminus\{0\}$ if $P_{v,\rm N} = \cQ$, $P_{v,\rm D} = 0$, $P_{v,\rm R} = \cP$ and $\Lambda_v$ is the multiplication by $\frac{d}{\beta_v}$.   
    \end{enumerate}
\end{definition}

The vertex conditions given in Definition \ref{def:Cond} can be written more explicitly. In fact, $f$ satisfies $\delta$ vertex conditions with strength $\alpha_v$ at a vertex $v$ if and only if $F(v)$ is equal to a constant vector, whose value we denote $f (v)$, and 
  \[\sum_{j=1}^{\deg (v)} F_j'(v) = \alpha_v f(v).   
  \]
In particular, continuity-Kirchhoff coupling conditions equal $\delta$ coupling conditions with strength $\alpha_v=0$. On the other hand, formally setting $\alpha_v=\infty$ results in a Dirichlet condition at $v$, i.e. $f(v)=0$.
  
Furthermore, $f$ satisfies $\delta'$ coupling conditions with strength $\beta_v$ at $v$ if and only if $F'(v)$ is equal to a constant vector, whose value we denote $f'(v)$, and 
  \[\sum_{j=1}^d F_j(v) = \beta_v f'(v).
  \]
In particular, an anti-Kirchhoff coupling condition can be interpreted as a $\delta'$ coupling condition with strength $\beta_v=0$. Formally setting $\beta_v=\infty$ results in a Neumann condition at $v$, i.e. $f'(v)=0$.

For $\delta$  and $\delta'$ couplings, the vertex term in the quadratic form $h$ in Proposition~\ref{prop:SAconditions} looks as follows.

\begin{lemma}[{e.g.\ \cite[Lemma 2.5]{RS20}}]
\label{lem:forms}
Let $\Gamma$ be a compact metric graph, let $H$ be a self-adjoint Schr\"odinger operator in $L^2 (\Gamma)$ as in Proposition~\ref{prop:SAconditions}, and let $h$ be the corresponding quadratic form. Furthermore, let $v \in \cV$ and let the vertex conditions for $H$ at $v$ be given in terms of $P_{v, \rm D}, P_{v, \rm N}, P_{v, \rm R}$ and $\Lambda_v$. 
% Assume that Hypothesis~\ref{hyp:PQ} holds for~$v$. 
Then the following assertions hold for each $f \in \dom h$.
\begin{enumerate}
  \item If a continuity-Kirchhoff condition is imposed at $v$, then $f$ is continuous at $v$ and
		\begin{align*}
	 \big\langle \Lambda_v P_{v, \rm R} F (v), P_{v, \rm R} F (v) \big\rangle
 = 0.
\end{align*}
	\item If a $\delta$ coupling condition with strength $\alpha_v$ is imposed at $v$, then $f$ is continuous at $v$ and
	\begin{align*}
	 \big\langle \Lambda_v P_{v, \rm R} F (v), P_{v, \rm R} F (v) \big\rangle
 = \alpha_v |f (v)|^2.
\end{align*}
  \item If an anti-Kirchhoff condition is imposed at $v$, then $\sum\limits_{j = 1}^{\deg (v)} F_j (v) = 0$ and
		\begin{align*}
	 \big\langle \Lambda_v P_{v, \rm R} F (v), P_{v, \rm R} F (v) \big\rangle
 = 0.
\end{align*}
	\item If a $\delta'$ coupling condition with strength $\beta_v$ is imposed at $v$, then $f$ does not satisfy any vertex conditions at $v$ and
	\begin{align*}
	 \big\langle \Lambda_v P_{v, \rm R} F (v), P_{v, \rm R} F (v) \big\rangle
 = \frac{1}{\beta_v} \bigg| \sum_{j = 1}^{\deg (v)} F_j (v) \bigg|^2.
\end{align*}
\end{enumerate}
\end{lemma}

\section{Properties of the ground state}
\label{sec:groundState}

In this section we review properties of the first eigenvalue and the corresponding eigenfunctions in case that the same coupling conditions are imposed at each vertex, either $\delta$ or $\delta'$. 

\subsection{Positivity of the first eigenfunction}

Here we discuss Perron-Frobenius (or Courant) type properties of quantum graphs with $\delta$ and $\delta'$ vertex conditions. For, e.g., Schrödinger operators on Euclidean domains with suitable boundary conditions, it is well known that the first eigenvalue has multiplicity one and the corresponding eigenfunction can be chosen positive. This turns out to be correct for $\delta$ coupling conditions, but is no longer true for $\delta'$ conditions, as we will see below.

The following theorem corresponds to \cite[Theorem~3.2]{EJ12}, see also \cite[Corollary 1]{K19}; for the convenience of the reader, we provide a proof.

\begin{theorem}\label{thm:courant}
Let $\Gamma$ be a connected compact metric graph, let $q \in L^\infty (\Gamma)$ be real-valued and let $H$ be the Schrödinger operator on $\Gamma$ defined in Proposition \ref{prop:SAconditions} and assume that $\delta$ coupling conditions are imposed at all vertices. Then the eigenspace of $H$ corresponding to the first eigenvalue $\lambda_1 (H)$ is one-dimensional. Moreover, the corresponding eigenfunction does not have any zero in $\Gamma$ and, hence, can be chosen positive.
\end{theorem}

\begin{proof}
Let $f \in \ker (H - \lambda_1 (H))$, $f \neq 0$, and let $g := |f|$. Then $|g'| \leq |f'|$ and therefore
\begin{align*}
 h (g) & = \int_\Gamma |g'|^2 + \int_\Gamma q |g|^2 + \sum_{v \in \cV} \alpha_v |g(v)|^2 \\
 & \leq \int_\Gamma |f'|^2 + \int_\Gamma q |f|^2 + \sum_{v \in \cV} \alpha_v |f (v)|^2 = h (f),
\end{align*}
and as $f$ minimizes the Rayleigh quotient, the same is true for $g$. Hence $g$ is a non-trivial element of $\ker (H - \lambda_1 (H))$; in particular, $g$ satisfies $\delta$ coupling conditions. Note that $g$ does not vanish on any vertex. Indeed, $g (0) = 0$ and the coupling conditions imply 
\begin{align*}
 \sum_{j = 1}^{\deg (v)} G_j' (v) = 0,
\end{align*}
and since $g$ vanishes at $v$ and is nonnegative everywhere, it follows $g_e' (v) =  0$ for each edge $e$ incident to $v$. In particular, $g$ vanishes on each edge incident to $v$, and it follows $g (w) = 0$ for each vertex $w$ which is connected to $v$ by an edge. Successively the same argument yields $g = 0$ identically on $\Gamma$ as $\Gamma$ is connected, a contradiction. As every non-vertex point on $\Gamma$ can be interpreted as a vertex of degree two equipped with $\delta$ conditions of strength zero, it follows that $g$ has no zero on $\Gamma$, and this implies that the arbitrarily chosen eigenfunction $f$ has no zero. 

Finally, having two linearly independent eigenfunctions $f, k$ would imply that for any chosen $x_0 \in \Gamma$ the function
\begin{align*}
 f - \frac{f (x_0)}{k (x_0)} k
\end{align*}
belongs to $\ker (H - \lambda_1 (H))$, is nontrivial and vanishes at $x_0$, a contradiction. Thus $\ker (H - \lambda_1 (H))$ is one-dimensional.
\end{proof}

The case of $\delta'$ coupling conditions is fundamentally different, as the following example shows. 

% \begin{theorem}
% Theorem~\ref{thm:courant} is valid for $\delta'$ coupling conditions if $\beta_v < 0$ holds for each $v \in V$.
% \end{theorem}
% 
% \begin{proof}
% Let $f$ be an eigenfunction, that is, $f$ minimizes the Rayleigh quotient for $H$. Then with $g := |f|$ we have $g \in \dom h$, $\int_\Gamma |g|^2 = \int_\Gamma |f|^2$, and
% \begin{align*}
%  h (g) & = \int_\Gamma |g'|^2 + \int_\Gamma q |g|^2 \textup{d}x + \sum_{v \in V} \frac{1}{\beta_v} \Big| \sum_{e \in E_v} g_e (v) \Big|^2 \\
%  & = \int_\Gamma |f'|^2 + \int_\Gamma q |f|^2 \textup{d}x + \sum_{v \in V} \frac{1}{\beta_v} \Big| \sum_{e \in E_v} |f_e (v)| \Big|^2 \\
%  & \leq \int_\Gamma |f'|^2 + \int_\Gamma q |f|^2 \textup{d}x + \sum_{v \in V} \frac{1}{\beta_v} \Big| \sum_{e \in E_v} f_e (v) \Big|^2 \\
%  & = h (f)
% \end{align*}
% by the triangle inequality. (Here the sign of the $\beta_v$ was used.) It follows that $g$ minimizes the Rayleigh quotient as well and, hence, is an eigenfunction. Moreover, $g$ vanishing at some vertex $v$ implies $\beta_v g' (v) = 0$, that is, $g' (v) = 0$ and thus $g_e' (v) = 0$ for each $e \in E_v$. The result is again that $g$ vanishes identically on each edge incident to $v$. If $w$ is a vertex connected to $v$ by an edge, it follows $g' (w) = 0$ and, hence $\sum_{e \in E_w} g_e (w) = 0$. But as $g$ is nonnegative, this implies $g_e (w) = 0$ for all $e \in E_w$. Successively we get $g = 0$ identically on $\Gamma$, a contradiction. The remaining proof is identic to the proof of Theorem~\ref{thm:courant}.
% \end{proof}

\begin{example}\label{ex:3star}
Consider a 3-star graph formed of three copies of the interval $[0, 1]$ coupled to each other at their zero endpoints, see Figure \ref{fig:3star}. 
\begin{figure}[htb]
  \centering
  \begin{tikzpicture}
    \draw[fill] (0,0) circle(0.05) node[below]{$0$};
    \draw[fill] (0,2) circle(0.05);
    \draw[fill] ({sqrt(3)},-1) circle(0.05);
    \draw[fill] ({-sqrt(3)},-1) circle(0.05);
    \draw (0,0)--(0,2);
    \draw (0,0)--({sqrt(3)},-1);
    \draw (0,0)--(-{sqrt(3)},-1);
  \end{tikzpicture}
  \caption{The equilateral 3-star graph.}
  \label{fig:3star}
\end{figure}
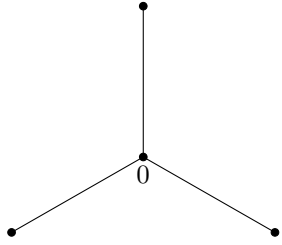
If we impose Neumann boundary conditions at the degree-one vertices (which can be read as $\delta'$ conditions with strength $\infty$) and a $\delta'$ coupling condition with a strength $\beta \geq 0$ at the degree-three vertex, then the corresponding quadratic form for the Laplacian $H$ is non-negative, i.e.\ $\lambda_1 (H) \geq 0$. Moreover, the function $f$ which is constantly $1$ on one of the edges, $- 1$ on another and zero on the remaining edge satisfies all imposed vertex conditions. As $  f'' = 0$ on every edge, we have $f \in \ker (H - \lambda_1 (H)) = \ker H$. But $f$ is positive on one edge and negative on another, that is, there exist ground state eigenfunctions without fixed sign, and they can even be constantly zero on a whole edge. Furthermore, by interchanging the roles of the edges, one sees that the eigenspace is two-dimensional.
\end{example}

For the study of a larger class of vertex conditions that allow for uniqueness and positivity of the ground state eigenfunction we refer the reader to \cite{K19}.

\subsection{Existence of negative eigenvalues}

For a quantum graph with $\delta$ or $\delta'$ vertex conditions, it is not necessarily true that existence of one negative coupling coefficient implies that the lowest eigenvalue is negative; cf.\ Example \ref{ex:noNegEV} below. In the following, sufficient conditions for existence of a negative eigenvalue are provided.  

\begin{theorem}\label{thm:negDelta}
Let $\Gamma$ be a compact metric graph and $H$ a self-adjoint Schrödinger operator on $\Gamma$ as defined in Proposition \ref{prop:SAconditions}. Assume that at each vertex $v \in \cV$ a $\delta$ coupling condition of strength $\alpha_v$ is imposed. If 
\begin{align*}
 \int_\Gamma q  + \sum_{v \in \cV} \alpha_v < 0,
\end{align*}
then $\lambda_1 (H) < 0$. 
\end{theorem}

\begin{proof}
By plugging the constant function $f = 1$ into the quadratic form we get
\begin{align*}
 h (f) & = \int_\Gamma |f'|^2 + \int_\Gamma q |f|^2  + \sum_{v \in \cV} \alpha_v |f (v)|^2 \\
 & = \int_\Gamma q  + \sum_{v \in \cV} \alpha_v < 0.
\end{align*}
This implies $\lambda_1 (H) < 0$.
\end{proof}

For $\delta'$ couplings, a more localized condition for existence of negative eigenvalues can be given. 

\begin{theorem}\label{thm:negDelta'}
Let $\Gamma$ be a compact metric graph and $H$ a self-adjoint Schrödinger operator on $\Gamma$ as defined in Proposition \ref{prop:SAconditions}. Assume that at each vertex $v \in \cV$ a $\delta'$ coupling condition of strength $\beta_v$ is imposed. If for some edge $e_0$, 
\begin{align*}
 \int_{e_0} q  + \frac{1}{\beta_{o (e_0)}} + \frac{1}{\beta_{t (e_0)}} < 0,
\end{align*}
then $\lambda_1 (H) < 0$. 
\end{theorem}

\begin{proof}
This time we may use the function $f$ being constantly equal to one on $e_0$ and zero otherwise. Then $f$ belongs to $\dom h$ and
\begin{align*}
 h (f) & = \int_\Gamma |f'|^2 + \int_\Gamma q |f|^2 + \sum_{v \in \cV} \frac{1}{\beta_v} \Big| \sum_{j=1}^{\deg(v)} F_j(v) \Big|^2 \\
 & = \int_{e_0} q  + \frac{1}{\beta_{o (e_0)}} + \frac{1}{\beta_{t (e_0)}} < 0.
\end{align*}
Thus $\lambda_1 (H) < 0$.
\end{proof}

\begin{example}
Let $\Gamma$ be a connected compact metric graph and $H$ the Laplacian (i.e.\ the Schrödinger operator with constant zero potential) on $\Gamma$ with $\delta'$ coupling conditions at each vertex $v$, where $\beta_v = 1$ for all $v \neq v_0$ and $\beta_{v_0} = - 1/2$, for a selected vertex $v_0$. If $e_0$ is any edge incident to $v_0$, then
\begin{align*}
 \frac{1}{\beta_{o (e_0)}} + \frac{1}{\beta_{t (e_0)}} = - 2 + 1 < 0.
\end{align*}
Therefore, Theorem \ref{thm:negDelta'} yields $\lambda_1 (H) < 0$.
\end{example}

The following example shows that the analogous condition for $\delta$ vertex conditions is not sufficient for existence of negative eigenvalues.

\begin{example}\label{ex:noNegEV}
Let $\Gamma$ be any compact metric graph. Consider again the case $q = 0$ constantly, and suppose that at each vertex $v \in \cV$ a $\delta$ coupling condition of strength $\alpha_v$ is imposed. Assume that for some $v_0 \in \cV$, $\alpha_{v_0} < 0$, while the remaining $\alpha_v$ are all positive. Our aim is to show that $h (f) \geq 0$ for all $f \in \dom h = \widetilde H^1 (\Gamma)$, supposed $\alpha_{v_0}$ is sufficiently close to zero, in relation to the lengths of its adjacent edges. Indeed, the Schrödinger operator $\widehat H$ obtained from $H$ by replacing the coupling coefficient $\alpha_{v_0}$ at $v_0$ by zero satisfies $\lambda_1 (\widehat H) > 0$. As the eigenvalue $\lambda_1 (H)$ depends continuously on the coupling coefficients, $\lambda_1 (H) > 0$ whenever $\alpha_{v_0} < 0$ is sufficiently close to zero. 
\end{example}

We finally point out that the number of negative eigenvalues of Schrödinger operators on metric graphs with general self-adjoint vertex conditions was studied in \cite{BL10}; cf.\ also \cite{H13}.

\section{Graph manipulations changing the total length of the graph}
\label{sec:surgery-principles_edges}

In the following we review the effect of certain geometric manipulations of a given metric graph onto the spectrum of a Schrödinger operator on the graph. We start with manipulations which change the total length of the metric graph $\Gamma$. These include increasing or shrinking the length of one or all edges, attaching new edges or graphs, or inserting a graph at a vertex.

\subsection{Increasing the length of one edge}
\label{subsec:Length_one}
Here, we consider the change of the eigenvalues in case we increase the length of one edge. For the Laplacian with continuity-Kirchhoff, $\delta$ or Dirichlet vertex conditions, the following theorem can be found in \cite[Corollary 3.12]{BKKM19}.

\begin{theorem}
\label{thm:edge_length}
  Let $\Gamma$ be a compact metric graph, $q\in L^\infty(\Gamma)$ real-valued and $H$ the Schr\"odinger operator on $\Gamma$ with potential $q$ and arbitrary self-adjoint coupling conditions. 
  Moreover, let $\widetilde \Gamma$ be the metric graph obtained from $\Gamma$ by increasing the length of one edge, i.e.\ there exists $e_0\in \cE$ such that $\widetilde L (e) = L (e)$ for $e \neq e_0$ and $\widetilde L (e_0) > L (e_0)$, where $\widetilde L$ is the length function for $\widetilde \Gamma$. Set $\widetilde q_e:=q_e$ for $e\neq e_0$ and $\widetilde{q}_{e_0}:=q_{e_0}(\frac{L(e_0)}{\widetilde L(e_0)} \cdot)$. Denote by $\widetilde H$ the Schr\"odinger operator on $\widetilde \Gamma$ with potential $\widetilde{q}$ equipped with the same coupling conditions as for $H$. Then
  \begin{align*}
    \lambda_k (\widetilde H) \leq \lambda_k (H)
  \end{align*}
  holds for all $k \in \N$ such that $\lambda_k (H) \geq 0$. If all coupling conditions are of $\delta$ type and $\lambda_1 (H) > 0$, i.e.\ the spectrum is positive, then
  \begin{align*}
    \lambda_1 (\widetilde H) < \lambda_1 (H).
  \end{align*}
\end{theorem}

\begin{proof}
  Let $h$ be the quadratic form associated with $H$ and $\widetilde h$ the quadratic form associated with $\widetilde H$; cf.\ Proposition \ref{prop:SAconditions}.
  Let $k\in\N$ such that $\lambda_k (H) \geq 0$ and let $F \subseteq \dom \form$ be a $k$-dimensional subspace such that
  \begin{align*}
    h (f) \leq \lambda_k (H) \int_\Gamma |f|^2 \quad \text{for all}~f \in F.
  \end{align*}
  Let $x_0 \in [0, L (e_0)]$.
  %, and identify the corresponding edge in the new graph with the concatenation of the intervals $[0, x_0]$, $[0, \widetilde L (e) - L (e)]$ and $[x_0, L (e)]$. 
  For each $f \in F$ define a function $\widetilde f$ on $\widetilde \Gamma$ by $\widetilde f_e := f_e$ for $e \neq e_0$ and 
  \[\widetilde f_{e_0}(x):= \begin{cases}
			      f_{e_0}(x), & x\in [0,x_0],\\
			      f_{e_0}(x_0), & x\in (x_0,x_0+\widetilde{L}(e_0)-L(e_0)],\\
			      f_{e_0}\bigl(x-\bigl(\widetilde{L}(e_0)-L(e_0)\bigr)\bigr), & x\in (x_0+\widetilde{L}(e_0)-L(e_0),\widetilde{L}(e_0)].
                            \end{cases}\]                            
  Then $\widetilde F (v) = F (v)$ for each $v \in \cV$, in particular, $\widetilde f \in \dom \widetilde h$, $\widetilde h(\widetilde f) = h(f)$. Moreover,
  \begin{align}\label{eq:yeah}
  \int_{\widetilde \Gamma} |\widetilde f|^2 & = \int_\Gamma |f|^2  + (\widetilde L (e_0) - L (e_0)) |f_{e_0} (x_0)|^2 \geq \int_\Gamma |f|^2, \quad f \in F.
  \end{align}
  For all $f \in F$ such that $h (f) < 0$, clearly $\frac{ \widetilde h (\widetilde f)}{\int_{\widetilde \Gamma} |\widetilde f|^2} < 0 \leq \lambda_k (H)$, and for all $f \in F$ with $h (f) \geq 0$ it follows
  \begin{align*}
  \frac{\widetilde h (\widetilde f)}{\int_{\widetilde \Gamma} |\widetilde f|^2} \leq \frac{h (f)}{\int_\Gamma |f|^2} \leq \lambda_k.
  \end{align*}
  Hence $\lambda_k (\widetilde H) \leq \lambda_k (H)$. 

  If each vertex is equipped with a $\delta$ coupling condition, then $\ker (H - \lambda_1 (H))$ is one-dimensional and the corresponding eigenfunction has no zeroes on $\Gamma$, see Theorem \ref{thm:courant}. In particular, $f_{e_0} (x_0) \neq 0$ and~\eqref{eq:yeah} implies $\int_{\widetilde \Gamma} |\widetilde f|^2 > \int_\Gamma |f|^2 $. Hence, if $\lambda_1 (H) > 0$ then $\lambda_1 (\widetilde H) < \lambda_1 (H)$.
\end{proof}

The restriction to non-negative eigenvalues in the above theorem is necessary for some coupling conditions such as $\delta$  or $\delta'$ couplings, as the following examples show (see also \cite[Theorem 4.1 and Section 5]{EJ12}).

\begin{example}
We consider a loop graph, i.e.\ a graph consisting of one vertex $v$ and one edge $e$ of length $L := L (e)$ originating from and terminating at $v$, see Figure~\ref{fig:loop}. 
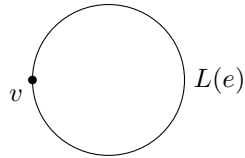
\begin{figure}[htb]
  \centering
  \begin{tikzpicture}
    \draw[fill] (2,0) circle(0.05) node[below left]{$v$};
    \draw (3,0) circle(1);
    \draw (4,0) node[right]{$L(e)$};
  \end{tikzpicture}
  \caption{The loop graph.}
  \label{fig:loop}
\end{figure}
At $v$ we impose a $\delta$ coupling condition of strength $\alpha := \alpha_v < 0$. Then the only negative eigenvalue of $H$ is given by the number $\lambda = - k^2$, where $k$ is the positive solution of
\begin{align*}
 2 k - \alpha \sinh (k L) - 2 k \cosh (k L) = 0,
\end{align*}
the corresponding eigenfunction being a linear combination of $x\mapsto \sinh (k x)$ and $x\mapsto \cosh (k x)$. It can be seen that $k$ is decreasing when $L$ increases, and thus $\lambda$ increases.

Exactly the same behavior is displayed by $\delta'$ coupling conditions. In the same setting, for a $\delta'$ condition of strength $\beta := \beta_v < 0$, the negative eigenvalues are given by $\lambda = - k^2$, where $k$ is the positive solution of
\begin{align*}
 \beta k \sinh (k L) + 2 \cosh (k L) + 2 = 0,
\end{align*}
and again, $\lambda$ increases when $L$ is increased.
\end{example}

\begin{example}
The following example was studied in \cite[Section 5]{EJ12}.    Let $\Gamma$ be a star graph, i.e.\ a graph consisting of a central vertex $v_0$ and a number of vertices $v_1, \dots, v_r$ of degree one, each of which is connected to $v_0$ by one edge. In this case, let us assume that $\Gamma$ is a star with three edges two of which have length $l_1:=1$ and the third of which has length $l_2:=l>0$. At the degree-one vertices we impose $\delta$ coupling conditions, with strength $\alpha_1:=-3/2$ for the edges of length $1$, and $\alpha_2:=-2$ for the edge of length $l$. At the central vertex we impose a strength $\alpha<0$, see Figure \ref{fig:star3}.
\begin{figure}[htb]
  \centering
  \begin{tikzpicture}
    \draw[fill] (0,0) circle(0.05) node[below]{$\alpha$};
    \draw[fill] (0,2) circle(0.05) node[above]{$\alpha_2=-2$};
    \draw[fill] ({sqrt(3)},-1) circle(0.05) node[below]{$\alpha_1=-3/2$};
    \draw[fill] ({-sqrt(3)},-1) circle(0.05) node[below]{$\alpha_1=-3/2$};;
    \draw (0,0)--(0,2);
    \draw (0,0)--({sqrt(3)},-1);
    \draw (0,0)--(-{sqrt(3)},-1);
    \draw (0,1) node[left]{$l$};
    \draw({0.5*sqrt(3)},-0.5) node[above right]{$l_1=1$};
    \draw({-0.5*sqrt(3)},-0.5) node[above left]{$l_1=1$};
  \end{tikzpicture}
  \caption{The star graph.}
  \label{fig:star3}
\end{figure}
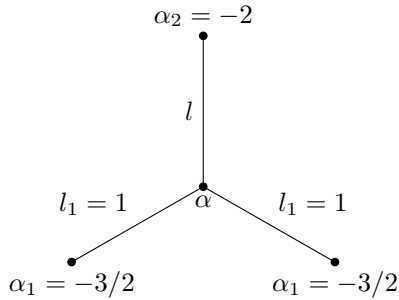
    Let $H$ be the Laplacian on $\Gamma$ corresponding to these vertex conditions. Then $\lambda_1(H)<0$ by Theorem \ref{thm:negDelta}. In \cite[Section 5]{EJ12} it was shown that a constant $\alpha_c \approx 1.09$ with the following properties exists:
    \begin{itemize}
     \item if $\alpha_{c}<\alpha<0$, then $\lambda_1(H)$ decreases is $l$ increases;
     \item if $\alpha=\alpha_c$ then $\lambda_1(H)$ stays constant as $l$ increases;
     \item if $\alpha<\alpha_c$, then $\lambda_1(H)$ increases as $l$ increases.
    \end{itemize}
\end{example}

Finally, we point out that strictness of the inequality in Theorem \ref{thm:edge_length} cannot be guaranteed in the case of $\delta'$ coupling conditions, as the following example shows.

\begin{example}
Consider the metric graph and vertex conditions of Example \ref{ex:3star}. There the edge lengths do not play any role and any changes of them will keep $\lambda_1 (H) = 0$ fixed. 
\end{example}

\subsection{Increasing the length of all edges}
\label{subsec:lengthAll}
Here, we consider the change of the eigenvalues in case we increase the length of all edges simultaneously. Iterating the results in Subsection \ref{subsec:Length_one}, we easily obtain that the positive eigenvalues decrease if all edge lengths are increased. More can be said if all edge lengths are increased by the same factor and also the coupling conditions are adjusted according to the factor. At least for continuity-Kirchhoff and $\delta$ vertex conditions the following theorem is folklore. However, for the sake of completeness we present a proof.
\begin{theorem}
\label{thm:scaling_length}
    Let $\Gamma$ be a compact metric graph, $q\in L^\infty(\Gamma)$ real-valued, and let $H$ be the Schr\"odinger operator on $\Gamma$ with potential $q$ and general self-adjoint vertex conditions as described in Proposition \ref{prop:SAconditions}. Let $t>0$ and obtain $\widetilde{\Gamma}$ by scaling each edge, the potential and the strengths, i.e.\ $\widetilde \Gamma$ has the length function $\widetilde{L}(e):=tL(e)$ and the potential $\widetilde{q}_e:=\frac{1}{t^2} q_e(\frac{\cdot}{t})$  for $e\in \cE$, as well as $\widetilde \Lambda_v = \frac{1}{t} \Lambda_v$ for $v\in \cV$. Let $\widetilde{H}$ be the corresponding Schr\"odinger operator on $\widetilde{\Gamma}$. Then
    \[\lambda_k(\widetilde{H}) = \frac{1}{t^2} \lambda_k(H)\]
    holds for all $k\in\N$.
\end{theorem}

\begin{proof}
Let $h$ and $\widetilde{h}$ be the quadratic forms associated with $H$ and $\widetilde{H}$, respectively. For $f \in \widetilde H^1 (\Gamma)$ and its restriction $f_e$ to an arbitrary edge $e \in \cE$, parametrised as $[0, L (e)]$, consider the function $\widetilde f \in \widetilde H^1 (\widetilde \Gamma)$ defined on each edge $e$ by
\begin{align*}
 \widetilde{f_e} (x) := f_e \left( \frac{x}{t} \right), \quad x \in [0, t L (e)].
\end{align*}
Since the boundary values of $\widetilde f_e$ and $f_e$ at each vertex are the same, i.e., $\widetilde F (v) = F (v)$ for each $v \in \cV$, $f$ belongs to $\dom h$ if and only if $\widetilde f$ belongs to $\dom \widetilde h$. By the substitution $y = \frac{x}{t}$ we observe
\begin{align*}
 \widetilde{h}(\widetilde{f}) & = \sum_{e\in \cE}\int_{0}^{t L(e)} |\widetilde f'|^2 + \sum_{e\in \cE} \int_0^{t L(e)} \widetilde{q}_e |\widetilde{f}_e|^2 + \sum_{v \in \cV} \big\langle \widetilde \Lambda_v P_{v, \rm R} \widetilde F (v), P_{v, \rm R} \widetilde F (v) \big\rangle\\
 & = \frac{1}{t} \sum_{e\in \cE}\int_{0}^{ L(e)} | f'|^2 + \frac{1}{t}\sum_{e\in \cE} \int_0^{{L}(e)} {q}_e |{f}_e|^2 + \frac{1}{t} \sum_{v \in \cV} \big\langle \Lambda_v P_{v, \rm R} F (v), P_{v, \rm R} F (v) \big\rangle \\
 & = \frac{1}{t} h(f).
\end{align*}
Moreover, $\int_{\widetilde \Gamma} |\widetilde{f}|^2 = t \int_{\Gamma} |f|^2$, by the same substitution. Thus, the min-max principle yields
\begin{align*}
 \lambda_k (\widetilde{H}) = \min_{\substack{\widetilde F \subseteq \dom \widetilde h\\ \dim \widetilde F = k}} \max_{\widetilde f \in \widetilde F} \frac{\widetilde h (\widetilde f)}{\int_{\widetilde \Gamma} |\widetilde f|^2} = \min_{\substack{F \subseteq \dom h\\ \dim F = k}} \max_{f \in F} \frac{1}{t^2} \frac{h (f)}{\int_{\Gamma} |f|^2} = \frac{1}{t^2} \lambda_k (H)
\end{align*}
for all $k\in\N$. 
\end{proof}

We emphasize the special cases of $\delta$  and $\delta'$ coupling conditions in the next corollary; observe that the coupling strengths for $\delta$ and $\delta'$ conditions have to be scaled differently.

\begin{corollary}
Let $\Gamma$ be a compact metric graph, $q\in L^\infty(\Gamma)$ real-valued and let $H$ be the Schr\"odinger operator on $\Gamma$ with potential $q$ such that at each vertex $v \in \cV$ either a $\delta$ coupling condition with strength $\alpha_v$ or a $\delta'$ coupling conditions with strength $\beta_v$ is imposed. Let $t>0$ and obtain $\widetilde{\Gamma}$ by scaling each edge, the potential and the strengths, i.e.\ $\widetilde{L}(e):=tL(e)$ and $\widetilde{q}_e:=\frac{1}{t^2} q_e(\frac{\cdot}{t})$  for $e\in \cE$, and $\widetilde{\alpha}_v:=\frac{1}{t}\alpha_v$ respectively $\widetilde{\beta}_v:=t\beta_v$ for each $v\in \cV$. Let $\widetilde{H}$ be the corresponding Schr\"odinger operator on $\widetilde{\Gamma}$. Then
    \[\lambda_k(\widetilde{H}) = \frac{1}{t^2} \lambda_k(H)\]
holds for all $k\in\N$.
\end{corollary}

\begin{proof}
 This follows directly from Theorem \ref{thm:scaling_length} and, at each vertex $v$, $\Lambda_v = \frac{\alpha_v}{\deg(v)}$ in the case of a $\delta$ coupling, respectively $\Lambda_v = \frac{\deg(v)}{\beta_v}$ in the case of a $\delta'$ coupling.
\end{proof}

\subsection{Attaching an edge between two vertices}

Next we focus on the mani\-pulation in which a new edge with finite length between two vertices $v_1,v_2\in \cV$ is added, see Figure \ref{fig:THMmonotonicity}. 
\begin{figure}[htb]
  \centering
  \begin{tikzpicture}
    \draw[fill] (0,0) circle(0.05);
    \draw[fill] (2,0) circle(0.05);
    \draw[fill] (1,1.5) circle(0.05) node[left]{$v_1$};
    \draw[fill] (4,0) circle(0.05);
    \draw[fill] (3,1.5) circle(0.05) node[right]{$v_2$};
    \draw (1,1.5)--(0,0)--(2,0);
    \draw (2,0)--(4,0)--(3,1.5);
    \draw (1,1.5) arc(180:247.38:1.625);
    \draw (3,1.5) arc(0:-67.38:1.625);
    \begin{scope}[shift={(7,0)}]
      \draw[fill] (0,0) circle(0.05);
      \draw[fill] (2,0) circle(0.05);
      \draw[fill] (1,1.5) circle(0.05) node[left]{$v_1$};
      \draw[fill] (4,0) circle(0.05);
      \draw[fill] (3,1.5) circle(0.05) node[right]{$v_2$};
      \draw (1,1.5)--(0,0)--(2,0);
      \draw (2,0)--(4,0)--(3,1.5);
      \draw (1,1.5) arc(180:247.38:1.625);
      \draw (3,1.5) arc(0:-67.38:1.625);
      \draw[semithick] (1,1.5)--(3,1.5);
    \end{scope}
  \end{tikzpicture}
  \caption{Attaching a new edge between the vertices $v_1$ and $v_2$.}
  \label{fig:THMmonotonicity}
\end{figure}
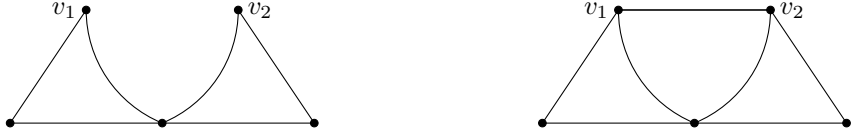
For the case of $\delta'$ coupling conditions at $v_1$ and $v_2$, this may only decrease the eigenvalues of $H$; this can be proven by simply extending test functions in the Rayleigh quotient by zero on the respective new edges.

\begin{theorem}[{\cite[Theorem 3.2]{RS20}}]
\label{thm:monotonicity}
Let $\Gamma$ be a compact metric graph, $q\in L^\infty(\Gamma)$ real, and let $v_1, v_2\in \cV$ be two distinct vertices of $\Gamma$. Furthermore, let $H$ be the Schr\"odinger operator on $\Gamma$ with potential $q$ subject to arbitrary self-adjoint coupling conditions at each vertex $v \in \cV \setminus \{v_1, v_2\}$, and having $\delta'$ coupling conditions of strengths $\beta_{v_1}$ respectively $\beta_{v_2}$ at $v_1$ and $v_2$; we allow the case of strengths zero, i.e.\ anti-Kirchhoff conditions. Let $\widetilde \Gamma$ be the graph obtained from $\Gamma$ by adding an extra edge $\widetilde{e}$ of finite length connecting $v_1$ and $v_2$ with potential $q_{\widetilde{e}}\in L^\infty((0,L(\widetilde{e})))$. Moreover, let $\widetilde H$ be the Schr\"odinger operator in $L^2 (\widetilde \Gamma)$ having the same potentials on the edges and the same coupling conditions as $H$ on all $v \in \cV \setminus \{v_1, v_2\}$, as well as the same strengths for the $\delta'$ couplings at $v_1$ and $v_2$, respectively. Then
\begin{align*}
  \lambda_k (\widetilde H) \leq \lambda_k (H)
\end{align*}
holds for all $k \in \N$.
\end{theorem}

\begin{proof}
We provide a proof for the case $\beta_{v_1} \neq 0 \neq \beta_{v_2}$; the case of an anti-Kirchhoff condition at one or both of the vertices $v_1, v_2$ is analogous. Let $h$ be the quadratic form associated with $H$ and $\widetilde h$ be the quadratic form associated with $\widetilde H$. Furthermore, let $F$ be a $k$-dimensional subspace of $\dom h$ such that
 \[h(f) \leq \lambda_k(H) \int_\Gamma |f|^2 \qquad \text{for all } f\in F.\]
 For $f\in F$ let $\widetilde{f}$ be the extension of $f$ by zero on $\widetilde{e}$, and $\widetilde{F}$ the space of all these $\widetilde{f}$. Then $\widetilde{F}$ is $k$-dimensional and $\widetilde{F}\subseteq\dom \widetilde{h}$.
 For $\widetilde f\in\widetilde F$ we observe
 \begin{align*}
  \widetilde{h}(\widetilde{f}) & = \int_\Gamma |f'|^2 + \int_\Gamma q|f|^2 + \sum_{v\in \cV\setminus\{v_1,v_2\}} \langle \Lambda_{v} P_{v,\rm R}F(v),P_{v,\rm R}F(v)\rangle \\
  & \quad + \frac{1}{\beta_{v_1}} \bigg| \sum_{j = 1}^{\deg (v_1)} F_j (v_1) \bigg|^2 + \frac{1}{\beta_{v_2}} \bigg| \sum_{j = 1}^{\deg (v_2)} F_j (v_2) \bigg|^2 = h (f).
 \end{align*}
%  Now, exploiting the coupling conditions at $v_1$ and $v_2$, cf.\ Lemma \ref{lem:forms}, we obtain
%  \[\langle \widetilde\Lambda_{v_j} \widetilde P_{v_j,\rm R}\widetilde F(v_j),\widetilde P_{v_j,\rm R}\widetilde F(v_j)\rangle = \langle \Lambda_{v_j}  P_{v_j,\rm R} F(v_j), P_{v_j,\rm R} F(v_j)\rangle \quad\text{for }j\in\{1,2\}.
%  \]
 Thus,
 \[\widetilde{h}(\widetilde{f}) = h(f) \leq \lambda_k(H) \int_\Gamma |f|^2 = \lambda_k(H) \int_{\widetilde{\Gamma}} |\widetilde f|^2 \quad\text{for all }\widetilde{f}\in \widetilde{F}.\]
 The min-max principle yields the assertion.
\end{proof}

The assertion of Theorem \ref{thm:monotonicity} is wrong for $\delta$ coupling conditions, and no definite statement is possible in this case, as the next example shows.

\begin{example}[{\cite[Example 1]{KMN13} and \cite[Example 3.3]{RS20}}]
\label{ex:KHdelta}
Let $\Gamma$ be the metric graph with two vertices $v_1, v_2$ and one edge of length~$1$ connecting these two, see Figure~\ref{fig:monotonicity}. 
Let $H$ be the Laplacian in $L^2 (\Gamma)$ with a $\delta$ condition of strength $\alpha \in \R$ at $v_1$ and a Kirchhoff condition at $v_2$. 
Note that both vertices have degree one; hence, the condition at $v_1$ is a Robin boundary condition and the one at $v_2$ is Neumann. 
Note further that for $\alpha = 0$ the condition at $v_1$ is Neumann, too. 
Moreover, let $\widetilde \Gamma$ be the graph obtained from $\Gamma$ by adding another edge of length $\ell > 0$ that also connects $v_1$ to $v_2$, and let $\widetilde H$ be the Laplacian in $L^2 (\widetilde \Gamma)$ subject to the natural extensions of the conditions for $H$, namely a $\delta$ condition of strength $\alpha$ at $v_1$ and a Kirchhoff condition at $v_2$. We look at two different cases.

  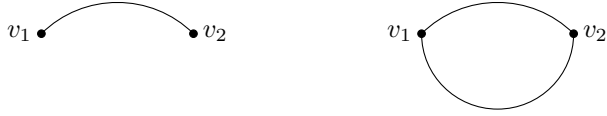
\begin{figure}[htb]
    \centering
    \begin{tikzpicture}
      \draw[fill] (0,0) circle(0.05);
      \draw[fill] (2,0) circle(0.05);
      \draw (0,0) node[left]{$v_1$};
      \draw (2,0) node[right]{$v_2$};
      \draw (2,0) arc(45:135:1.41421356237);
      \begin{scope}[shift={(5,0)}]
	\draw[fill] (0,0) circle(0.05);
	\draw[fill] (2,0) circle(0.05);
	\draw (0,0) node[left]{$v_1$};
	\draw (2,0) node[right]{$v_2$};
	\draw (2,0) arc(45:135:1.41421356237);
	\draw (2,0) arc(0:-180:1);
      \end{scope}
    \end{tikzpicture}     
    \caption{The graphs in Example~\ref{ex:KHdelta}. Left: $\Gamma$, right: $\widetilde{\Gamma}$.}
  \label{fig:monotonicity}
  \end{figure}

If $\alpha = 0$ (see~\cite[Example 1]{KMN13}) then $\lambda_k(H) = \pi^2(k-1)^2$ for all $k\in\N$, and $\widetilde{H}$ is unitarily equivalent to the Laplacian on an interval of length $1+\ell$ with periodic boundary conditions, which implies $\lambda_1(\widetilde{H}) = 0$ and $\lambda_{2k}(\widetilde{H}) = \lambda_{2k+1}(\widetilde{H}) = \frac{4\pi^2}{(1+\ell)^2} k^2$ for all $k\in\N$. 
Hence for $\ell<1$ we observe $\lambda_2(\widetilde{H}) > \lambda_2(H)$, whereas for $\ell\geq1$ we obtain $\lambda_k(\widetilde{H}) \leq \lambda_k(H)$ for all $k\in\N$ (and even a strict inequality for $k>2$; for $\ell>1$ the inequality is strict also for $k=2$).

If $\alpha = 1$ then simple calculations yield $\lambda_1(H) \approx 0.74017$. Furthermore, letting $\ell = 0.1$ one obtains $\lambda_1 (\widetilde{H}) \approx 0.83156$, that is, $\lambda_1 (H) < \lambda_1 (\widetilde H)$. Hence, adding an edge may increase the eigenvalues also for $\alpha \neq 0$.
\end{example}

\subsection{Attaching a pendant graph}

In this section we study how the eigenvalues behave if we attach a whole new graph to one vertex of another given graph. In contrast to the above Theorem \ref{thm:monotonicity}, extending test functions in the Rayleigh quotient constantly to the new edge is admissibile for both $\delta$  and $\delta'$ couplings. The result includes, in particular, the case of attaching a single pendant edge.

\begin{theorem}[{cf.\ \cite[Theorem 3.10]{BKKM19}, \cite[Theorem 3.5]{RS20}}]
\label{thm:attaching_graph}
Let $\Gamma$ be a compact metric graph, $q\in L^\infty(\Gamma)$ real and let $v\in \cV$ be a vertex of $\Gamma$. Furthermore, let $H$ be the Schr\"odinger operator on $\Gamma$ with potential $q$ subject to arbitrary self-adjoint vertex conditions at each vertex $w \in \cV \setminus \{v\}$, and with a $\delta$  or $\delta'$ coupling condition at $v$. Let $\widehat\Gamma$ be a compact metric graph, $\widehat q\in L^\infty(\widehat\Gamma)$, $\widehat{v}\in \widehat \cV$ be a vertex of $\widehat \Gamma$ and choose arbitrary self-adjoint vertex conditions at each vertex $w \in \widehat \cV \setminus \{\widehat v\}$. Let $\widetilde \Gamma$ be the graph obtained from $\Gamma$ by attaching the graph $\widehat \Gamma$ with vertex $\widehat v$ at $v$ by identifying $v$ and $\widehat{v}$, see Figure \ref{fig:thm_attaching_graph}.
Let $\widetilde H$ be the Schr\"odinger operator in $L^2 (\widetilde \Gamma)$ having the same vertex conditions as $H$ on all $w \in \cV\setminus \{v\}$ and the chosen coupling conditions at $w\in\widehat{\cV}\setminus\{\widehat{v}\}$, and with a $\delta$  or $\delta'$ coupling condition at $v$ with the same strength as for $H$. If the condition at $v$ is of $\delta$ type we assume in addition that the condition of $\widetilde H$ at $w\in\widehat{\cV}\setminus\{\widehat{v}\}$ are $\delta$ coupling conditions with strengths $\alpha_{w} \leq 0$ for all $w\in \widehat{\cV}\setminus\{\widehat{v}\}$ and that $\int_{\widehat \Gamma} \widehat q \leq 0$. Then
\begin{align*}%\label{eq:dasWollenWir}
  \lambda_k (\widetilde H) \leq \lambda_k (H) 
\end{align*}
holds for all $k \in \N$. 
\end{theorem}

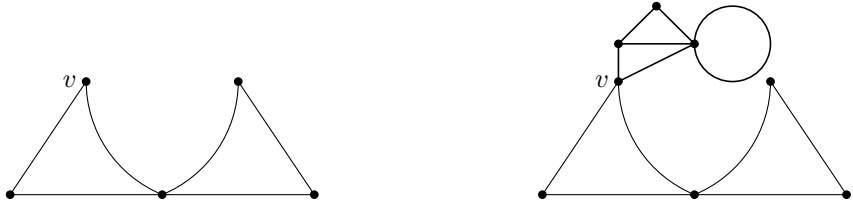
\begin{figure}[htb]
  \centering
  \begin{tikzpicture}
    \draw[fill] (0,0) circle(0.05);
    \draw[fill] (2,0) circle(0.05);
    \draw[fill] (1,1.5) circle(0.05) node[left]{$v$};
    \draw[fill] (4,0) circle(0.05);
    \draw[fill] (3,1.5) circle(0.05);
    \draw (1,1.5)--(0,0)--(2,0);
    \draw (2,0)--(4,0)--(3,1.5);
    \draw (1,1.5) arc(180:247.38:1.625);
    \draw (3,1.5) arc(0:-67.38:1.625);
    %\draw (0,0) arc(45:405:0.6);
    \begin{scope}[shift={(7,0)}]
      \draw[fill] (0,0) circle(0.05);
      \draw[fill] (2,0) circle(0.05);
      \draw[fill] (1,1.5) circle(0.05) node[left]{$v$};
      \draw[fill] (4,0) circle(0.05);
      \draw[fill] (3,1.5) circle(0.05);
      \draw (1,1.5)--(0,0)--(2,0);
      \draw (2,0)--(4,0)--(3,1.5);
      \draw (1,1.5) arc(180:247.38:1.625);
      \draw (3,1.5) arc(0:-67.38:1.625);
      %\draw (0,0) arc(45:405:0.6);
      \draw[fill] (2,2) circle(0.05);
      \draw[fill] (1.5,2.5) circle(0.05);
      \draw[fill] (1,2) circle(0.05);
      \draw[semithick] (1,1.5)--(2,2);
      \draw[semithick] (1,1.5)--(1,2);
      \draw[semithick] (1,2)--(2,2);
      \draw[semithick] (1.5,2.5)--(2,2);
      \draw[semithick] (1,2)--(1.5,2.5);
      \draw[semithick] (2.5,2) circle(0.5);
    \end{scope}
  \end{tikzpicture}
  \caption{The transformation of Theorem \ref{thm:attaching_graph}. Left: $\Gamma$, right: $\widetilde{\Gamma}$, where the graph $\widehat{\Gamma}$ (thick) is attached at $v$.}
  \label{fig:thm_attaching_graph}
\end{figure}

\begin{proof}    
 Let $h$ be the quadratic form associated with $H$ and $\widetilde h$ be the quadratic form associated with $\widetilde H$. Furthermore, let $F$ be a $k$-dimensional subspace of $\dom h$ such that
 \[h(f) \leq \lambda_k(H) \int_\Gamma |f|^2 \qquad \text{for all } f\in F.\]
 For $f\in F$ let $\widetilde{f}$ be the extension of $f$ by $0$ to $\widetilde{\Gamma}$ (i.e.\ $\widetilde{f} = 0$ on $\widehat\Gamma$) in the case of a $\delta'$ coupling at $v$, and by the constant value $f(v)$ to $\widetilde{\Gamma}$ (i.e.\ $\widetilde{f} = f(v)$ on $\widehat\Gamma$) in the case of a $\delta$ coupling at $v$. Moreover, let $\widetilde{F}$ be the space consisting of the extensions $\widetilde{f}$ of all $f \in F$. Then $\widetilde{F}$ is $k$-dimensional and $\widetilde{F}\subseteq\dom \widetilde{h}$.
 For $\widetilde f\in\widetilde F$ we observe
 \begin{align*}
 \widetilde{h}(\widetilde{f}) & = \int_\Gamma |f'|^2 + \int_\Gamma q|f|^2 + \int_{\widehat\Gamma} \widehat{q} |\widetilde{f}_{\widehat\Gamma}|^2 + \sum_{w\in \cV\setminus\{v\}} \langle \Lambda_{w} P_{w,\rm R}F(w),P_{w,\rm R}F(w)\rangle \\
 & \quad + \sum_{w\in\widehat\cV \setminus \{\widehat v\}}
  \langle \widehat\Lambda_{w} \widehat P_{w,\rm R}\widehat F(w),\widehat P_{w,\rm R}F(w)\rangle \\
  & \quad + \begin{cases} \frac{1}{\beta_v} \sum_{j = 1}^{\deg (v)} \widetilde F_j (v)  & \text{for a $\delta'$ condition at $v$},\\ \alpha_v |f (v)|^2 & \text{for a $\delta$ condition at $v$}. \end{cases}
 \end{align*}
 Note that
 \[\int_{\widehat\Gamma} \widehat {q} |\widetilde{f}|^2 \begin{cases} = 0 & \text{for a $\delta'$ condition at $v$},\\
 \leq 0 & \text{for a $\delta$ condition at $v$}.\end{cases}\] 
%  Now, exploiting the coupling conditions at $v$ and at the vertices of $\widehat{v}\in\widehat{\cV}$, cf.\ Lemma \ref{lem:forms}, we obtain
%  \[\langle \widetilde\Lambda_{v} \widetilde P_{v,\rm R}\widetilde F(v),\widetilde P_{v,\rm R}\widetilde F(v)\rangle = \langle \Lambda_{v}  P_{v,\rm R} F(v), P_{v,\rm R} F(v)\rangle,
%  \]
%  and
%  \[\langle \widehat\Lambda_{\widehat v} \widehat  P_{\widehat v,\rm R}\widehat  F(\widehat v),\widehat  P_{\widehat v,\rm R}\widehat  F(\widehat v)\rangle = \begin{cases} 0 & \text{$\delta'$ condition at $v$},\\
%  \alpha_{\widehat v}|f(v)|^2 & \text{$\delta$ condition at $v$} \end{cases} \leq 0.\]  
Moreover, note that the vertex terms at vertices $w \in \widehat \cV \setminus \{\widehat v\}$ are
 \[\langle \widehat\Lambda_{w} \widehat P_{w,\rm R}\widehat F(w),\widehat P_{w,\rm R}F(w)\rangle \begin{cases} = 0 & \text{for a $\delta'$ condition at $v$},\\
 \leq 0 & \text{for a $\delta$ condition at $v$},\end{cases}\]
 according to the assumptions of the theorem. Thus,
 \[\widetilde{h}(\widetilde{f}) \leq  h(f) \leq \lambda_k(H) \int_\Gamma |f|^2 \leq \lambda_k(H) \int_{\widetilde{\Gamma}} |\widetilde f|^2 \quad\text{for all }\widetilde{f}\in \widetilde{F}.\]
 The min-max principle yields the assertion.
\end{proof}

If we have a $\delta$ type condition at the gluing vertex $v$ then a condition on the coupling coefficients at the further vertices of the newly attached graph is necessary, as the following example shows.

\begin{example}[{\cite[Example 3.6]{RS20}}]
\label{ex:attach_new_vertex}
Let $\Gamma$ be the graph consisting of two vertices $v_0$ and $v_1$ and one edge connecting the two vertices, which is parametrised by the interval $[0,1]$; cf.~Figure~\ref{fig:attach_new_vertex}.
\begin{figure}[htb]
  \centering
  \begin{tikzpicture}
    \draw[fill] (0,0) circle(0.05) node[above]{$v_0$};
    \draw (0,0) node[below]{$0$};
    \draw[fill] (1,0) circle(0.05) node[above]{$v_1$};
    \draw (1,0) node[below]{$1$};  
    \draw (0,0)--(1,0);
    \begin{scope}[shift={(5,0)}]
     \draw[fill] (0,0) circle(0.05) node[above]{$v_0$};
      \draw (0,0) node[below]{$0$};
      \draw[fill] (1,0) circle(0.05) node[above]{$v_1$};
      \draw (1,0) node[below]{$1$}; 
      \draw (0,0)--(1,0);
      \draw[fill] (2,0) circle(0.05) node[above]{$v_2$};
      \draw (2,0) node[below]{$2$};
      \draw (1,0)--(2,0);
    \end{scope}
  \end{tikzpicture}
  \caption{The graphs in Example \ref{ex:attach_new_vertex}. Left: $\Gamma$, right: $\widetilde{\Gamma}$.}
  \label{fig:attach_new_vertex}
\end{figure}
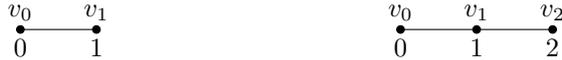
Let us impose continuity-Kirchhoff conditions at $v_0$ and~$v_1$. Since the degree is one in both cases, these conditions correspond to Neumann boundary conditions for $H$ at $0$ and at $1$. Hence, $\lambda_k (H) = (k-1)^2\pi^2$ for all $k\in\N$. Now, let us add an edge of length one, which connects $v_1$ to a new vertex $v_2$, see Figure \ref{fig:attach_new_vertex}.
At the vertex $v_2$ we impose either the condition $f' (2) + \alpha f (2) = 0$ for some $\alpha > 0$ or a Dirichlet condition. Then the spectrum of $\widetilde H$ is positive and, hence,
\begin{align*}
 \lambda_1 (\widetilde H) > 0 = \lambda_1 (H).
\end{align*}
\end{example}

\subsection{Inserting a graph at a vertex}

In this section we consider the graph manipulation of inserting a metric graph into a vertex of another given metric graph. This transformation was introduced in \cite[Section 3]{BKKM19} and its influence on the eigenvalues of the Laplacian with $\delta$ (and Dirichlet) vertex conditions was studied there. The operation may formally be defined as follows.

\begin{definition}
Let $v_0$ be one of the vertices of a compact metric graph $\Gamma$, and let $\Gamma'$ be another compact metric graph. Denote by $\cE_{v_0}$ the set of edges incident to $v_0$ in $\Gamma$. We say that the graph $\widetilde \Gamma$ is obtained by {\em inserting $\Gamma'$ into $\Gamma$ at $v_0$} if it is formed by removing $v_0$ from $\Gamma$ and attaching each of the edges $e \in \cE_{v_0}$ to one of the vertices of $\Gamma'$; cf.\ Figure \ref{fig:insertgraph}.
\end{definition}

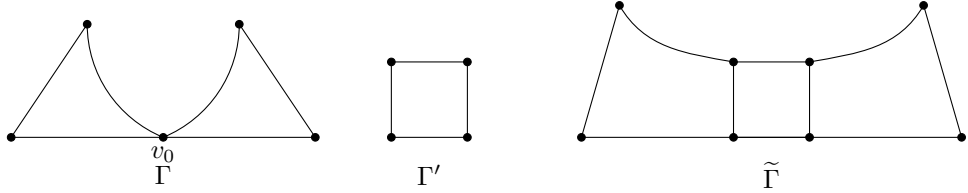
\begin{figure}[htb]
  \centering
  \begin{tikzpicture}
    \draw[fill] (0,0) circle(0.05);
    \draw[fill] (2,0) circle(0.05) node[below]{$v_0$};
    \draw[fill] (1,1.5) circle(0.05);
    \draw[fill] (4,0) circle(0.05);
    \draw[fill] (3,1.5) circle(0.05);
    \draw (1,1.5)--(0,0)--(2,0);
    \draw (2,0)--(4,0)--(3,1.5);
    \draw (1,1.5) arc(180:247.38:1.625);
    \draw (3,1.5) arc(0:-67.38:1.625);
    \draw (2,-0.5) node[]{$\Gamma$};
    
    \begin{scope}[shift={(5,0)}]
      \draw[fill] (0,0) circle(0.05);
      \draw[fill] (1,0) circle(0.05);
      \draw[fill] (1,1) circle(0.05);
      \draw[fill] (0,1) circle(0.05);
      \draw (0,0)--(1,0)--(1,1)--(0,1)--(0,0);
      \draw (0.5,-0.5) node[]{$\Gamma'$};
    \end{scope}

    \begin{scope}[shift={(7.5,0)}]
      \draw[fill] (0,0) circle(0.05);
      \draw[fill] (2,0) circle(0.05);
      \draw[fill] (0.5,1.75) circle(0.05);
      \draw[fill] (5,0) circle(0.05);
      \draw[fill] (4.5,1.75) circle(0.05);
      \draw[fill] (3,0) circle(0.05);
      \draw[fill] (3,1) circle(0.05);
      \draw[fill] (2,1) circle(0.05);
      \draw (2,0)--(3,0)--(3,1)--(2,1)--(2,0);
      \draw (0.5,1.75)--(0,0)--(2,0);
      \draw (2,0)--(5,0)--(4.5,1.75);
      \draw (0.5,1.75) to[out=-60, in=170] (2,1);
      \draw (4.5,1.75) to[out=240, in=10] (3,1);
      \draw (2.5,-0.5) node[]{$\widetilde\Gamma$};
    \end{scope}
  \end{tikzpicture}
  \caption{Inserting the graph $\Gamma'$ at the vertex $v_0\in \Gamma$.}
  \label{fig:insertgraph}
\end{figure}

With this operation, the following theorem holds for $\delta$ vertex conditions. It was shown for the potential-free case of the Laplacian in \cite[Theorem 3.10]{BKKM19}. Here we present a slightly different proof and allow the presence of a potential $q$ as well as more general vertex conditions at the vertices that are not affected by the manipulation.

\begin{theorem}
Let $\Gamma$ and $\Gamma'$ be compact metric graphs and let $\widetilde \Gamma$ be formed by inserting $\Gamma'$ into $\Gamma$ at a given vertex $v_0$ of $\Gamma$. Furthermore, let $H$ be a Schrödinger operator on $\Gamma$ with a real potential $q \in L^\infty (\Gamma)$, arbitrary self-adjoint vertex conditions at all vertices except $v_0$ and a $\delta$ vertex condition with strength $\alpha := \alpha_{v_0}$ at $v_0$. Consider the Schrödinger operator $\widetilde H$ on $\widetilde \Gamma$ with the same vertex conditions at all vertices of $\Gamma$ except $v_0$ and the same potential on all edges of $\Gamma$; on the edges of $\Gamma'$, potentials $q_e$ are placed such that $\int_{\Gamma'} q \leq 0$, and on the remaining vertex set $\cV' = \cV (\widetilde \Gamma) \setminus (\cV (\Gamma) \setminus \{ v_0 \})$, $\delta$ vertex conditions with coupling coefficients $\alpha_v$ are imposed such that 
\begin{align*}
 \sum_{v \in \cV'} \alpha_v \leq \alpha.
\end{align*}
Then 
\begin{align*}
 \lambda_k (\widetilde H) \leq \lambda_k (H) 
\end{align*}
holds for all $k \in \N$.
\end{theorem}

\begin{proof}
Let $F$ denote the linear span of the spaces $\ker (H - \lambda_j (H))$ for $1 \leq j \leq k$. Then $\dim F \geq k$, and
\begin{align*}
 h (f) \leq \lambda_k (H) \int_\Gamma |f |^2
\end{align*}
holds for all $f \in F$. Moreover, let $\widetilde F$ be the space formed by extending all $f \in F$ constantly to all of $\widetilde \Gamma$ by identifying $\Gamma$ as a subgraph of $\widetilde \Gamma$ in the natural way. Then also $\dim \widetilde F \geq k$, $\widetilde F$ belongs to the domain of the quadratic form $\widetilde h$ associated with $\widetilde H$, and for each $\widetilde f \in \widetilde F$ we have
\begin{align*}
 \widetilde h (\widetilde f) & = \left( h (f) - \alpha |f (v_0)|^2 \right) + |f (v_0)|^2 \left( \int_{\Gamma'} q + \sum_{v \in \cV'} \alpha_v \right) \\
 & \leq h (f) \leq \lambda_k (H) \int_\Gamma |f |^2  \leq \lambda_k (H) \int_{\widetilde \Gamma} |f|^2.
\end{align*}
This implies the assertion of the theorem.
\end{proof}

The case of a $\delta'$ vertex condition at the vertex $v_0$ at which the additional graph $\Gamma'$ is inserted is, on the one hand, less restrictive: arbitrary self-adjoint vertex conditions and arbitrary potentials on the edges may be imposed on the inserted graph. On the other hand, a sign condition on the coupling conditions at the affected vertices is required.

\begin{theorem}
\label{thm:insertDelta'}
Let $\Gamma$ and $\Gamma'$ be compact metric graphs and let $\widetilde \Gamma$ be formed by inserting $\Gamma'$ into $\Gamma$ at a given vertex $v_0$ of $\Gamma$. Furthermore, let $H$ be a Schrödinger operator on $\Gamma$ with real potential $q \in L^\infty (\Gamma)$, arbitrary self-adjoint vertex conditions at all vertices except $v_0$ and a $\delta'$ vertex condition with strength $\beta := \beta_{v_0} < 0$ at $v_0$. Consider the Schrödinger operator $\widetilde H$ on $\widetilde \Gamma$ with the same vertex conditions at all vertices of $\Gamma$ except $v_0$ and the same potential on all edges of $\Gamma$; on the edges of $\Gamma'$, arbitrary potentials $q_e$ are placed, on the vertex set $\widehat \cV$ obtained from splitting $v_0$, $\delta'$ vertex conditions with strengths $\beta_v < 0$ are imposed such that
\begin{align}\label{eq:sumBeta}
 \sum_{v \in \widehat \cV} \beta_v = \beta,
\end{align}
and on the remaining vertex set $\cV' = \cV (\widetilde \Gamma) \setminus (\cV (\Gamma) \setminus \{ v_0 \})$, arbitrary self-adjoint vertex conditions are imposed. Then 
\begin{align*}
 \lambda_k (\widetilde H) \leq \lambda_k (H) 
\end{align*}
holds for all $k \in \N$.
\end{theorem}

\begin{proof}
Let $\widehat \Gamma$ be the graph obtained from $\Gamma$ by splitting $v_0$ into a set of vertices $\widehat \cV$, and let $\widehat H$ be the Schrödinger operator on $\widehat \Gamma$ with the same potentials on the edges (after identification) and the same vertex conditions at all vertices in $\cV (\widehat \Gamma) \setminus \widehat \cV$. Moreover, at the new vertices in $\widehat \cV$, $\delta'$ vertex conditions with strengths $\beta_v < 0$ are imposed which satisfy \eqref{eq:sumBeta}. Then, by Theorem \ref{thm:joining_delta'} below, 
\begin{align*}
 \lambda_k (\widehat H) \leq \lambda_k (H)
\end{align*}
for all $k \in \N$. Furthermore, attaching $\Gamma'$ to $\widehat \Gamma$ in such a way that the resulting graph is $\widetilde \Gamma$, and imposing the desired vertex conditions leads to the operator $\widetilde H$ defined in the theorem, and 
\begin{align*}
 \lambda_k (\widetilde H) \leq \lambda_k (\widehat H)
\end{align*}
follows from a repeated application of Theorem \ref{thm:monotonicity} respectively Theorem \ref{thm:attaching_graph}. This completes the proof.
\end{proof}

\subsection{Shrinking edge lengths to zero}

Here we review some results from~\cite{BerkolaikoLatushkinSukhtaiev2019}, where shrinkage of edge lengths to zero has been considered. More precisely, in \cite{BerkolaikoLatushkinSukhtaiev2019} the natural limit vertex conditions when some edge lengths shrink to zero were identified, and convergence of the corresponding Schrödinger operators was shown under certain conditions. For simplicity of presentation we do restrict ourselves here to the potential-free case of the Laplacian, $q = 0$ identically. 

We change now slightly the perspective in the following sense. Up to now, we viewed Schr\"odinger operators on metric graphs as having coupling conditions at the vertices, where the boundary values of functions came from all edges adjacent to a given vertex. Now, we focus on each edge separately such that functions in $\widetilde{H}^2$ have two boundary values on each edge for the function as well as for the (inward) derivative, one for the left endpoint and one for the right endpoint, respectively. Since the sum of the vertex degrees is twice the number of edges for each graph, taking all boundary values at all vertices corresponds to taking all boundary values for all edges, so this change of perspective just yields a reordering of the boundary values (now sorted by edges and not by vertices).

In order to formulate a variant of the main result of \cite{BerkolaikoLatushkinSukhtaiev2019}, some notation has to be introduced.
For a compact metric graph $\Gamma$ with total number of edges $E =  |\cE|$, let $\cE_0 \subseteq \cE$ be a set of selected edges, which we are going to shrink to zero. For each edge $e$ denote by $l (e)$ and $r (e)$ the left, respectively right, end point of the interval with which $e$ is identified, and let
\begin{align*}
 \partial \Gamma := \bigcup_{e \in \cE} \{l (e), r (e) \};
\end{align*}
we distinguish all end points, even if they possibly correspond to the same points on the real line according to the parametrizations of the edges. Then for any function $f \in \widetilde H^1 (\Gamma)$ we may consider a vector
\begin{align*}
 f |_{\partial \Gamma} \in \C^{2 E}
\end{align*}
of evaluations of $f$ at all end points of all edges. Note that $f|_{\partial \Gamma}$ describes the same boundary values as $(F(v))_{v\in \cV}$, just sorted differently. Especially, if $f \in \dom H$ for some self-adjoint Schrödinger operator $H$ as in Proposition \ref{prop:SAconditions}, then the vertex conditions may be expressed in terms of the vectors $f |_{\partial \Gamma}$ and $f' |_{\partial \Gamma}$ where, in the latter case, the derivatives are taken in inward direction. Note as well that $f' |_{\partial \Gamma}$ describes the same boundary values for the derivatives as $(F'(v))_{v\in \cV}$, just sorted differently. Furthermore, define
\begin{align*}
    D_0 & := D_0(\cE_0) := \{F\in \C^{2 E}:\; F_{l (e)} = F_{r (e)},\; e\in \cE_0\},\\
    N_0 & := N_0(\cE_0) := \{F'\in \C^{2 E}:\; F'_{l(e)} = -F'_{r(e)},\; e\in \cE_0\}.
\end{align*}

In the following we will consider a decomposition of the edge set of $\Gamma$ into $\cE = \cE_+\cup \cE_0$ with $\cE_+ \cap \cE_0= \emptyset$ and shrink the lengths of the edges in $\cE_0$ to zero, while the edge lengths of the edges in $\cE_+$ may also change but will be sent to a positive limit. We let $\Gamma_+$ be the subgraph of $\Gamma$ constituted by the edge set $\cE_+$ and indicate the dependence of $\Gamma$ on the length function by $\Gamma (L)$; if $L$ is chosen formally such that $L (e) = 0$ for all $e \in \cE_0$, then $\Gamma (L) = \Gamma_+$.

Let us consider the sequence of self-adjoint operators $H (L)$, where $H (L)$ acts as the Laplacian on $\Gamma (L)$, $L (e) > 0$ for all $e \in \cE$, and all operators $H (L)$ are equipped with the same vertex conditions. The natural limit operator $\widetilde H$ of $H (L)$ if $L \to \widetilde L$, where $\widetilde L (e) = 0$ for all $e \in \cE_0$ and $\widetilde L (e) > 0$ for all $e \in \cE_+$ is the Laplacian defined on the restrictions to $\Gamma_+$ of all functions $f$ satisfying the chosen vertex conditions, such that, in addition,
\begin{align}\label{eq:limCond}
 f |_{\partial \Gamma} \in D_0 \quad \text{and} \quad f' |_{\partial \Gamma} \in N_0.
\end{align}

The following hypothesis on the vertex conditions of a self-adjoint Laplacian $H$ on a metric graph will be imposed for the next theorem.

\begin{hypothesis}\label{hyp}
Suppose that for all $f\in \dom H$ we have that if $f |_{\partial \Gamma} \in D_0$ and $f' |_{\partial \Gamma} \in N_0$ and if $f |_{\partial \Gamma_+} = f' |_{\partial \Gamma_+} = 0$, then $f |_{\partial \Gamma} = 0$.
\end{hypothesis}

For the special case of the Laplacian, the following theorem collects the results in \cite[Theorems 3.1, 3.5, 3.6]{BerkolaikoLatushkinSukhtaiev2019}, where the coupling conditions were formulated in terms of Lagrangian subspaces for the values at the vertices; see also \cite{SchubertSeifertVoigtWaurick2015}.

\begin{theorem}
% [{cf.\ \cite[Theorems 3.1, 3.5, 3.6]{BerkolaikoLatushkinSukhtaiev2019}}]
\label{thm:shrinkage}
Let $\Gamma = \Gamma (L)$ be a compact metric graph
% , $q\in L^\infty(\Gamma)$,
and let $H (L)$ be the Laplacian in $L^2 (\Gamma (L))$ subject to arbitrary, $L$-independent self-adjoint coupling conditions at each vertex. Let Hypothesis \ref{hyp} be satisfied
% , let $\widetilde{q}:= (q_e)_{e\in \cE_+}$,
and let $\widehat H$ be the Laplacian defined on the restrictions of functions satisfying the original vertex conditions as well as \eqref{eq:limCond}. Then $\widetilde{H}$ is self-adjoint, $H (L)$ converges to $\widetilde{H}$ in generalized norm resolvent sense as $L \to \widetilde L$, cf.\ Remark \ref{rem:gNRS}, and the spectrum of $H$ converges to the spectrum of $\widetilde{H}$ (in the Hausdorff sense of multisets) as the edge lengths of $\cE_0$ shrink to zero.
\end{theorem}

\begin{remark}\label{rem:Hyp}
Hypothesis \ref{hyp} was analyzed in \cite{BerkolaikoLatushkinSukhtaiev2019} in great detail. Amongst others, \cite[Lemma 3.3]{BerkolaikoLatushkinSukhtaiev2019} shows that among vertex conditions without Robin part (i.e.\ $P_{v, \rm R} = 0$ for all $v \in \cV$), Hypothesis \ref{hyp} is satisfied if and only if the zero function is the only function satisfying the prescribed vertex conditions and being both constant on each edge of $\Gamma_0$ and identically zero on $\Gamma_+$. Moreover, \cite[Lemma 3.4]{BerkolaikoLatushkinSukhtaiev2019} implies that Hypothesis \ref{hyp} is satisfied for all vertex conditions that require continuity at each vertex, i.e.\ for $\delta$ vertex conditions, including the special case of continuity-Kirchhoff conditions.
\end{remark}

\begin{remark}\label{rem:gNRS}
The convergence of $H(L)$ to $\widetilde H$ in generalized norm resolvent sense shown in Theorem \ref{thm:shrinkage} can be found in \cite[p.\ 351]{Weidmann2000}, see also \cite[Chapter 4]{P12}, and means that
\[\|J_L (\widetilde H+ \mathrm{i})^{-1} - (H(L)+\mathrm{i})^{-1} J_L\|\to 0\]
as $L\to \widetilde{L}$, where $\| \cdot \|$ denotes the norm in the space of bounded linear operators from $L^2 (\widetilde \Gamma)$ to $L^2 (\Gamma)$ and $J_L$ is the embedding of $L^2(\widetilde{\Gamma})$ to $L^2(\Gamma(L))$ given by
\[(J_L f)_e(x):=\begin{cases}
                    \sqrt{\frac{\widetilde{L}(e)}{L(e)}} f_e\Bigl(\frac{\widetilde{L}(e)}{L(e)} x\Bigr) & x\in (0,L(e)),\, e\in \cE_+,\\
                    0 & e\in \cE_0,
                \end{cases}\]
i.e.\ $J_L$ scales on the edges $\cE_+$ linearly and extends trivially by zero on the edges $\cE_0$.
\end{remark}

We illustrate the above result at two examples.

\begin{example}
% [Lasso Graph]
Consider a ``lasso graph'' consisting of an interval $e_1$ and a loop $e_2$; see Figure \ref{fig:lasso}. We impose anti-Kirchhoff coupling conditions; this simplifies to a Dirichlet condition at the degree-one vertex at the end of the interval.
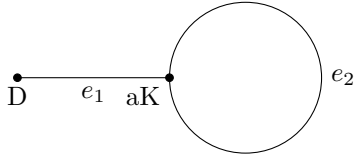
\begin{figure}[htb]
  \centering
  \begin{tikzpicture}
    \draw[fill] (0,0) circle(0.05) node[below]{D};
    \draw[fill] (2,0) circle(0.05) node[below left]{aK};
    \draw (0,0)--(2,0);
    \draw (1,0) node[below]{$e_1$};
    \draw (3,0) circle(1);
    \draw (4,0) node[right]{$e_2$};
  \end{tikzpicture}
  \caption{The lasso graph.}
  \label{fig:lasso}
\end{figure}
We send one or the other edge length to zero:
\begin{itemize}
 \item[(a)] Let $\cE_+ = \{e_1\}$ and $\cE_0 = \{e_2\}$. Then Hypothesis \ref{hyp} is satisfied. Indeed, in view of Remark \ref{rem:Hyp}, if $f=0$ on $e_1$ and $f$ is constant on $e_2$ (which in particular means $f_{e_2}(0) = f_{e_2}(L(e_2))$), then together with the anti-Kirchhoff condition at the vertex, i.e.\ $f_{e_2}(0) + f_{e_2}(L(e_2)) = 0$, we obtain $f=0$ on $e_2$.
 
 The limiting graph after shrinking the length of $e_2$ to zero is an interval graph with a Dirichlet condition on the left and a Neumann condition on the right.
 
 \item[(b)] Let $\cE_+ = \{e_2\}$ and $\cE_0 = \{e_1\}$. Then Hypothesis \ref{hyp} is also satisfied. Indeed, again with Remark \ref{rem:Hyp}, if $f=0$ on $e_2$ and $f$ is constant on $e_1$, the Dirichlet condition on the left endpoint of $e_1$ (as well as the anti-Kirchhoff coupling at the right endpoint of $e_1$) yields $f_{e_1}=0$.
 
 The limiting graph after shrinking the length of $e_1$ to zero is a loop graph with an anti-Kirchhoff conditionat the vertex.
\end{itemize}

The secular equation for the lasso graph is given by
\[2\cos(kL(e_1))\cos(kL(e_2)) + 2\cos(kL(e_1)) - \sin(kL(e_1))\sin(kL(e_2)) = 0,\]
so that the eigenvalues are given by $\lambda = k^2$, where $k$ is a solution of the secular equation.
% Now shinking one or the other edge as above results in the reduced secular equations
% \begin{itemize}
%  \item[(a)] $4\cos(kL(e_1)) = 0$,
%  \item[(b)] $2\cos(kL(e_2)) + 2 = 0$,
% \end{itemize}
The implicit function theorem yields that the solutions $k$ (and therefore the eigenvalues) converge when shrinking one edge. This resembles the spectral convergence in Theorem \ref{thm:shrinkage}.
\end{example}

\begin{example}
% [Circle Graph]
Take a circle graph with at least two vertices with anti-Kirchhoff coupling conditions as in Figure \ref{fig:circle_path}. We want to shrink one edge to zero, i.e.\ for one fixed $e_0 \in \cE$ we have $\cE_0 = \{e_0\}$ and $\cE_+ = \cE\setminus\{e_0\}$.
By Remark \ref{rem:Hyp} it is easy to see that Hypothesis \ref{hyp} is satisfied. Let us determine the coupling for the limiting graph. In order to do that, let us assume that the edges are parametrized counterclockwise, and let $\check{e}$ be the edge preceding $e_0$ and $\hat{e}$ be the edge following $e_0$. Then the anti-Kirchhoff condition at the left endpoint of $e_0$ read
\[-f'_{\check{e}}(L(\check{e})) = f_{e_0}' (0),\quad f_{\check{e}}(L(\check{e})) + f_{e_0} (0) = 0,\]
and analogously for the right endpoint of $e_0$ we obtain
\[-f'_{e_0}(L(e_0)) = f_{\hat{e}}'(0),\quad f_{e_0} (L(e_0)) + f_{\hat{e}}(0) = 0.\]
Moreover, having boundary values in $D_0$ yields
\[f_{e_0} (0) = f_{e_0} (L(e_0)),\]
while having boundary values for the (inward) derivatives in $N_0$ yields
\[f_{e_0} '(0) = f_{e_0}' (L(e_0)).\]
Together, we obtain that shinking the length of $e_0$ to zero yields the coupling conditions
\[f_{\check{e}}(L(\check{e})) = f_{\hat{e}}(0),\quad -f_{\check{e}}'(L(\check{e})) = - f_{\hat{e}}'(0),\]
thus we obtain a continuity-Kirchhoff condition relating the edges $\check{e}$ and $\hat{e}$.
\begin{figure}[htb]
  \centering
  \begin{tikzpicture}
    \draw[fill] (1,0) circle(0.05) node[right]{aK};
    \draw[fill] (-0.5,{0.5*sqrt(3)}) circle(0.05) node[above left]{aK};
    \draw[fill] (-0.5,{-0.5*sqrt(3)}) circle(0.05) node[below left]{aK};
    \draw (0,0) circle(1);
    \draw (0.5,{0.5*sqrt(3)}) node[above right]{$\check{e}$};
    \draw (-1,0) node[left]{$e_0$};
    \draw (0.5,{-0.5*sqrt(3)}) node[below right]{$\hat{e}$};
    
    \begin{scope}[shift={(7,0)}]
        \draw[fill] (-1,0) circle(0.05) node[left]{cK};
        \draw[fill] (1,0) circle(0.05) node[right]{aK};
        \draw (0,0) circle(1);
        \draw (0,1) node[above]{$\check{e}$};
        \draw (0,-1) node[below]{$\hat{e}$};
    \end{scope}
  \end{tikzpicture}
  \caption{Left: A circle graph with anti-Kirchhoff coupling conditions and the resulting circle graph after shrinking edge $e_0$ to zero length.}
  \label{fig:circle_path}
\end{figure}
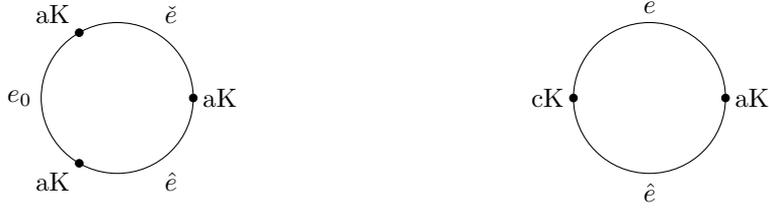
\end{example}

We point out that the theorem does not generally apply to graphs with anti-Kirchhoff or $\delta'$ vertex conditions, as the following example shows.

\begin{example}
Consider a ``pumpkin'' graph consisting of two vertices $v_1, v_2$ and three edges $e_1, e_2, e_3$ all connecting $v_1$ and $v_2$ and having arbitrary positive lengths, see Figure \ref{fig:pumpkin}. 
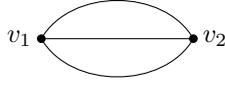
\begin{figure}[htb]
  \centering
  \begin{tikzpicture}
    \draw[fill] (0,0) circle(0.05) node[left]{$v_1$};
    \draw[fill] (2,0) circle(0.05) node[right]{$v_2$};
    \draw (0,0)--(2,0);
    \draw (0,0) to[out=60, in = 120] (2,0);
    \draw (0,0) to[out=-60, in = 240] (2,0);
  \end{tikzpicture}
  \caption{The pumpkin graph with three edges.}
  \label{fig:pumpkin}
\end{figure}
Impose either anti-Kirchhoff or $\delta'$ vertex conditions at $v_1$ and $v_2$, and let $\cE_0 = \{e_1, e_2\}$ and $\cE_+ = \{e_3\}$. Our ambition is, thus, to send $L (e_1)$ and $L (e_2)$ to zero. However, Theorem \ref{thm:shrinkage} is not applicable. Indeed, the function $f$ that is constantly $1$ on $e_1$, $-1$ on $e_2$ and zero on $e_3$ belongs to $\dom (H)$, i.e.\ it satisfies the imposed vertex conditions. Moreover, on each edge in $\cE_0$, the function has the same value at both end points and zero derivatives, so that $f |_{\partial \Gamma} \in D_0$ and $f' |_{\partial \Gamma} \in N_0$ and $f |_{\partial \Gamma_+} = f' |_{\partial \Gamma_+} = 0$. But, clearly, $f$ has non-zero boundary values on the edges $e_1$ and $e_2$; thus Hypothesis \ref{hyp} is violated.
\end{example}

For further illuminating examples we refer the reader to \cite[Section 3]{BerkolaikoLatushkinSukhtaiev2019}.

\section{Graph manipulations preserving the total length}
\label{sec:surgery-principles_vertices}

In this section we consider graph manipulations that keep the total length constant and study their influence on the eigenvalues of Schrödinger operators. In some cases also the coupling conditions at the vertices are changed.

\subsection{Changing the strength of a coupling}

We start by considering the effect of a change of the strength of a coupling condition. We formulate a general statement for coupling conditions with non-trivial Robin part. The cases of $\delta$  and $\delta'$ couplings will then be simple corollaries.

\begin{theorem}
\label{thm:strength_general}
  Let $\Gamma$ be a compact metric graph, $q\in L^\infty(\Gamma)$ real, $v_0\in \cV$ and $H$ the Schr\"odinger operator on $\Gamma$ with potential $q$, arbitrary self-adjoint coupling conditions as in Proposition \ref{prop:SAconditions} at the vertices $v\in \cV\setminus\{v_0\}$ and a self-adjoint coupling condition with non-trivial Robin part, $P_{v_0, \rm R} \neq 0$, at the vertex $v_0$. Denote by $\Lambda_{v_0}$ the self-adjoint, invertible coupling operator in $\ran P_{v_0, \rm R}$ and let $\widetilde \Lambda_{v_0}$ be another self-adjoint, invertible operator in $\ran P_{v_0, \rm R}$ such that $\widetilde \Lambda_{v_0} \leq \Lambda_{v_0}$.
%   and $\Lambda_{v_0} - \widetilde \Lambda_{v_0}$ has a non-trivial kernel. 
  Denote by $\widetilde H$ the Schr\"odinger operator on $\Gamma$ obtained from $H$ by replacing the coupling operator $\Lambda_{v_0}$ at $v_0$ by $\widetilde \Lambda_{v_0}$. Then 
  \begin{align}\label{eq:strengthEst}
    \lambda_k (\widetilde H) \leq \lambda_k (H)
  \end{align}
  holds for all $k \in \N$. If the eigenvalue $\lambda_k (H)$ is simple and for the corresponding eigenfunction $f$ and its associated vector $F (v_0)$ of boundary evaluations at the vertex~$v_0$, $P_{v_0, \rm R} F (v_0) \notin \ker (\Lambda_{v_0} - \widetilde \Lambda_{v_0})$, then 
   \begin{align*}
    \lambda_k (\widetilde H) < \lambda_k (H).
  \end{align*}
\end{theorem}

\begin{proof}
Let $h$ be the quadratic form associated with $H$ and $\widetilde h$ be the quadratic form associated with $\widetilde H$. Moreover, denote by $F$ the span of the spaces $\ker (H - \lambda_j (H))$ for $1 \leq j \leq k$. Then $\dim F \geq k$, and for all $f \in F$ we have
\begin{align}\label{eq:smallDifference}
 h (f) - \widetilde h (f) & = \big\langle \Lambda_{v_0} P_{v_0, \rm R} F (v), P_{v_0, \rm R} F (v) \big\rangle - \big\langle \widetilde \Lambda_{v_0} P_{v_0, \rm R} F (v), P_{v_0, \rm R} F (v) \big\rangle \geq 0.
\end{align}
Thus, the assertion \eqref{eq:strengthEst} follows immediately from the min-max principle. 

Assume now that for some $k$, $\lambda_k (H)$ is simple and equality holds in \eqref{eq:strengthEst}. Then
\begin{align*}
 \lambda_{k - 1} (\widetilde H) \leq \lambda_{k - 1} (H) < \lambda_k (H) = \lambda_k (\widetilde H);
\end{align*}
in particular, for all nontrivial $f$ in the space $F_{k-1}$ formed of the span of $\ker (H - \lambda_j (H))$ for $1 \leq j \leq k - 1$,
\begin{align*}
 \widetilde h (f) \leq \lambda_{k - 1} (H) \int_\Gamma |f|^2 < \lambda_k (H) \int_\Gamma |f|^2.
\end{align*}
Therefore the assumption $\lambda_k (\widetilde H) = \lambda_k (H)$ implies equality in \eqref{eq:smallDifference} for all $f \in \ker (H - \lambda_k (H))$. This implies
\begin{align*}
 P_{v_0, \rm R} F (v_0) \in \ker (\Lambda_{v_0} - \widetilde \Lambda_{v_0}).
\end{align*}
for all $f \in \ker (H - \lambda_1 (H))$ and proves the second assertion.
\end{proof}

For the cases of $\delta$  and $\delta'$ coupling conditions at $v_0$, we obtain the following two corollaries.

\begin{corollary}[{\cite[Theorem 3.1.8]{BK13}}]
%\label{cor:strength_alpha}
  Let $\Gamma$ be a compact metric graph, $q\in L^\infty(\Gamma)$ real, $v_0\in \cV$ and let $H$ be the Schr\"odinger operator on $\Gamma$ with potential $q$, arbitrary self-adjoint coupling conditions at the vertices $v\in \cV\setminus\{v_0\}$ and a $\delta$ coupling condition with strength $\alpha_{v_0}\in\R$ at $v_0$. Let $\widetilde\alpha_{v_0}<\alpha_{v_0}$ and denote by $\widetilde H$ the Schr\"odinger operator on $\Gamma$ with the same potential as $H$, the same coupling conditions at all $v\in \cV\setminus\{v_0\}$ and a $\delta$ coupling condition with strength $\widetilde \alpha_{v_0}$ at $v_0$. Then 
  \begin{align*}
    \lambda_k (\widetilde H) \leq \lambda_k (H)
  \end{align*}
  holds for all $k \in \N$. If the eigenvalue $\lambda_k (H)$ is simple and $f (v_0) \neq 0$, then 
   \begin{align*}
    \lambda_k (\widetilde H) < \lambda_k (H).
  \end{align*}
  In particular, $\lambda_1 (\widetilde H) < \lambda_1 (H)$ if $\delta$ coupling conditions are imposed at all vertices.
\end{corollary}

\begin{proof}
This follows immediately from Theorem \ref{thm:strength_general}, since $\Lambda_{v_0} = \frac{\alpha_{v_0}}{\deg (v_0)}$ and functions in $\dom H$ are continuous in this case.
\end{proof}

Thus the eigenvalues are monotone in each strength for $\delta$ couplings. However, the situation is more involved for $\delta'$ couplings. Note that the following theorem includes the case of a $\delta'$ coupling of strength zero, i.e.\ an anti-Kirchhoff condition.

\begin{corollary}
% [{cf.\ \cite[Theorem 3.1]{AliUsman2022}}]\label{cor:strength_beta}
  Let $\Gamma$ be a compact metric graph, $q\in L^\infty(\Gamma)$ real, $v_0\in \cV$ and let $H$ be the Schr\"odinger operator on $\Gamma$ with potential $q$, arbitrary self-adjoint coupling conditions at the vertices $v\in \cV\setminus\{v_0\}$ and a $\delta'$ coupling condition with strength $\beta_{v_0}\in\R$ at $v_0$. Let $\widetilde\beta_{v_0}\in\R$. Denote by $\widetilde H$ the Schr\"odinger operator on $\Gamma$ with potential $q$ with the same coupling conditions as $H$ at all $v\in \cV\setminus\{v_0\}$ and a $\delta'$ coupling condition with strength $\widetilde \beta_{v_0}$ at $v_0$. Suppose that one of the following conditions is satisfied:
  \begin{enumerate}
  \item[(a)] $0<\beta_{v_0}<\widetilde \beta_{v_0}$,
  \item[(b)] $\beta_{v_0}<\widetilde \beta_{v_0}<0$,
  \item[(c)] $\widetilde \beta_{v_0} < 0 < \beta_{v_0}$,
  \item[(d)] $\widetilde \beta_{v_0} < \beta_{v_0} = 0$.
  \end{enumerate}
  Then
  \begin{align*}
    \lambda_k (\widetilde H) \leq \lambda_k (H)
  \end{align*}
  holds for all $k \in \N$.
%   \item 
%   If $\beta_{v_0} <0<\widetilde \beta_{v_0}$ or $\widetilde \beta_{v_0} = 0$, then
%   \begin{align*}
%     \lambda_k (H) \leq \lambda_k (\widetilde H)
%   \end{align*}
%   for all $k \in \N$.  
%   \end{enumerate}
\end{corollary}

\begin{proof}
If $\beta_{v_0},\widetilde\beta_{v_0} \neq 0$, the statements follow immediately from Theorem \ref{thm:strength_general} and $\Lambda_{v_0} = \frac{\deg{v_0}}{\beta_{v_0}}$. For the case of assumption (d), the quadratic forms associated with $H$ and $\widetilde H$ satisfy $\dom h \subseteq \dom \widetilde h$ and $h(f) = \widetilde{h}(f)$ for all $f\in \dom h$. Thus, the min-max principle yields the assertion in this case.
\end{proof}

\subsection{Changing \texorpdfstring{$\delta$}{delta} couplings to \texorpdfstring{$\delta'$}{delta'} couplings}

Closely related to the change of the coupling strength at a vertex considered in the previous section is replacing a $\delta$ coupling condition by a $\delta'$ coupling condition. 
% For a semibounded self-adjoint operator $H$ with discrete spectrum let $N_-(H)$ denote the number of negative eigenvalues.

\begin{theorem}\label{thm:laaangweilig}
Let $\Gamma$ be a compact metric graph, $q\in L^\infty(\Gamma)$ real, and let $v_0 \in \cV$. Denote by $H$ the Schrödinger operator with potential $q$ subject to arbitrary self-adjoint vertex conditions at all $v \in \cV \setminus \{v_0\}$ and a $\delta$ vertex condition with coupling strength $\alpha_{v_0} \in \R$ at $v_0$. Furthermore, denote by $\widetilde H$ the Schrödinger operator with the same potential and the same vertex conditions as for $H$ at all vertices except $v_0$, and a $\delta'$ vertex condition with strength $\beta_{v_0} \in \R \setminus \{0\}$ at $v_0$. Assume that
  \begin{align*}
    \frac{\deg (v_0)^2}{\beta_{v_0}} \leq \alpha_{v_0}.
  \end{align*}
Then
\begin{align*}
 \lambda_k (\widetilde H) \leq \lambda_k (H)
\end{align*}
holds for all $k \in \N$.
\end{theorem}

A version of this theorem for Lipschitz partitions of the Euclidean space can be found in \cite[Corollary 4.3, Corollary 4.9]{BEV14}.

\begin{proof}
Let $h$ and $\widetilde h$ be the quadratic forms associated with $H$ and $\widetilde H$, respectively, and let $f \in \dom h$. Then $f \in \dom \widetilde h$ and
\begin{align*}
    h (f) - \widetilde h (f) & = \alpha_{v_0} |f (v_0)|^2 - \frac{1}{\beta_{v_0}} \Bigl| \sum_{j=1}^{\deg(v_0)} F_j (v_0) \Bigr|^2 \\
    & = \alpha_{v_0} |f (v_0)|^2 - \frac{\deg (v_0)^2}{\beta_{v_0}} | f (v_0) |^2 \geq 0,
\end{align*} 
which implies the assertion.
\end{proof}

The case of an anti-Kirchhoff condition, $\beta_{v_0} = 0$, is more involved; the above proof fails since in this case no inclusion between the domains of the two quadratic forms $h$ and $\widetilde h$ prevails. The following example illustrates what can happen.

\begin{example}
Consider a star graph with three edges of arbitrary lengths and the zero potential on them. Impose Neumann vertex conditions at the degree-one vertices. For the operator $H$, we impose a continuity-Kirchhoff vertex condition at the central vertex, while $\widetilde H$ is equipped with an anti-Kirchhoff condition there. As earlier, any function which is constantly $1$ on one edge, $-1$ on another and $0$ on the third edge belongs to $\ker \widetilde H$, and these functions span a two-dimensional vector space. We obtain
\begin{align*}
 \lambda_2 (\widetilde H) = 0 < \lambda_2 (H).
\end{align*}

Consider, on the other hand, the Laplacian on an interval and fix a Neumann boundary condition at the left end point for both $H$ and $\widetilde H$. At the right end point we impose a Kirchhoff (i.e.\ Neumann) condition for $H$ and an anti-Kirchhoff (i.e.\ Dirichlet) condition for $\widetilde H$. Then 
\begin{align*}
 \lambda_1 (H) = 0 < \frac{\pi^2}{4 L^2} = \lambda_1 (\widetilde H),
\end{align*}
where $L$ denotes the length of the interval.
\end{example}

Complementary to this example, it should be mentioned that in certain situations continuity-Kirchhoff and anti-Kirchhoff vertex conditions lead to the same positive eigenvalues; cf.\ \cite{KR20}.

\subsection{Joining two vertices}

Given a graph, we may join two vertices by identifying them. More specifically, the graph $\widetilde \Gamma$ is obtained from $\Gamma$ by replacing two vertices $v_1, v_2$ by a new vertex $v_0$ with $\deg (v_0) = \deg (v_1) + \deg (v_2)$, such that
\begin{align*}
 \cE_{v_0} = \cE_{v_1} \cup \cE_{v_2};
\end{align*}
see Figure \ref{fig:THMjoining}. We review now how the eigenvalues of Schrödinger operators behave under this transformation.
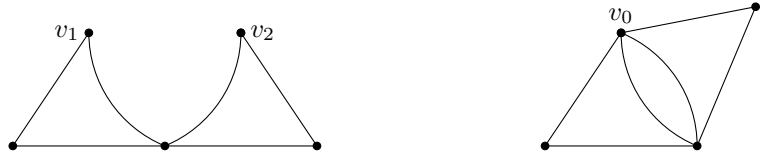
\begin{figure}[htb]
  \centering
  \begin{tikzpicture}
    \draw[fill] (0,0) circle(0.05);
    \draw[fill] (2,0) circle(0.05);
    \draw[fill] (1,1.5) circle(0.05) node[left]{$v_1$};
    \draw[fill] (4,0) circle(0.05);
    \draw[fill] (3,1.5) circle(0.05);
    \draw[fill] (3,1.5) circle(0.05) node[right]{$v_2$};
    \draw (1,1.5)--(0,0)--(2,0);
    \draw (2,0)--(4,0)--(3,1.5);
    \draw (1,1.5) arc(180:247.38:1.625);
    \draw (3,1.5) arc(0:-67.38:1.625);
    %\draw (0,0) arc(45:405:0.6);
    \begin{scope}[shift={(7,0)}]
      \draw[fill] (0,0) circle(0.05);
      \draw[fill] (2,0) circle(0.05);
      \draw[fill] (1,1.5) circle(0.05) node[above]{$v_0$};
      \draw (1,1.5)--(0,0)--(2,0);
      \draw (1,1.5) arc(180:247.38:1.625);
      \begin{scope}[rotate around={67.38:(2,0)}]
	\draw (3,1.5) arc(0:-67.38:1.625);
	\draw[fill] (4,0) circle(0.05);
	\draw[fill] (3,1.5) circle(0.05);
	\draw (2,0)--(4,0)--(3,1.5);
      \end{scope}
      %\draw (0,0) arc(45:405:0.6);
    \end{scope}
  \end{tikzpicture}
  \caption{Joining the vertices $v_1$ and $v_2$ to a new vertex $v_0$.
  }
  \label{fig:THMjoining}
\end{figure}
We start with the case of $\delta$ coupling conditions at the vertices to be joined.

\begin{theorem}[{\cite[Theorem~2]{KKT16}, 
% \cite[Theorem~4.1]{RS20}
\cite[Theorem 5.3]{BK12}}]
\label{thm:joining}
Let $\Gamma$ be a compact metric graph, $q\in L^\infty(\Gamma)$ real, $v_1,v_2\in \cV$ with $v_1\neq v_2$ and $H$ the Schr\"odinger operator on $\Gamma$ with potential $q$, arbitrary self-adjoint coupling conditions at the vertices $v\in \cV\setminus\{v_1,v_2\}$, and $\delta$ coupling conditions with strenghs $\alpha_{v_1},\alpha_{v_2}\in\R$ at $v_1$ and $v_2$, respectively. Denote by $\widetilde \Gamma$ the graph obtained from $\Gamma$ by joining $v_1$ and $v_2$ to form one single vertex $v_0$. Let $\widetilde H$ be the self-adjoint Schr\"odinger operator in $L^2 (\widetilde \Gamma)$ having the same potential and the same coupling conditions as $H$ at all vertices apart from $v_0$ and satisfying a $\delta$ coupling condition with strength $\alpha_{v_0} :=\alpha_{v_1} + \alpha_{v_2}$ at $v_0$. Then
\begin{align*}%\label{eq:yeah}
 \lambda_k (H) \leq \lambda_k (\widetilde H)
\end{align*}
holds for all $k\in\N$.
\end{theorem}

\begin{proof}
The quadratic forms $h$ and $\widetilde h$ associated with $H$ and $\widetilde H$, respectively, satisfy $\dom\widetilde{h}\subseteq \dom h$. In particular, each $f \in \dom \widetilde h$ satisfies $f (v_0) = f (v_1) = f (v_2)$, and by Lemma~\ref{lem:forms} we get
\begin{align*}
 \widetilde h (f) - h (f) & = \alpha_{v_0} |f (v_0)|^2 - \alpha_{v_1} |f (v_1)|^2 - \alpha_{v_2} |f (v_2)|^2 = 0.
\end{align*}
By the min-max principle, this leads to the assertion.
\end{proof}

While joining two $\delta$ couplings in the above way may only increase the eigenvalues, the situation is more involved for $\delta'$ couplings.

\begin{theorem}[{\cite[Theorem 4.2]{RS20}}]
\label{thm:joining_delta'}
Let $\Gamma$ be a compact metric graph, $q\in L^\infty(\Gamma)$ real, $v_1,v_2\in \cV$ with $v_1\neq v_2$ and let $H$ be the Schr\"odinger operator in $L^2(\Gamma)$ with potential $q$, arbitrary self-adjoint coupling conditions at the vertices $v\in \cV\setminus\{v_1,v_2\}$, and $\delta'$ coupling conditions with strenghs $\beta_{v_1},\beta_{v_2}\in\R$ at $v_1$ and $v_2$, respectively. Denote by $\widetilde \Gamma$ the graph obtained from $\Gamma$ by joining $v_1$ and $v_2$ to form one single vertex $v_0$. Furthermore, let $\widetilde H$ be the self-adjoint Schr\"odinger operator in $L^2 (\widetilde \Gamma)$ having the same potential and the same coupling conditions as $H$ at all vertices apart from $v_0$ and satisfying a $\delta'$ coupling condition with strength $\beta_{v_0} :=\beta_{v_1} + \beta_{v_2}$ at $v_0$.
\begin{enumerate}
 \item If $\beta_{v_1},\beta_{v_2} \geq 0$, or $\beta_{v_1}\cdot \beta_{v_2} < 0$ and $\beta_{v_0} < 0$, then $\lambda_k (\widetilde H) \leq \lambda_k (H)$ for all $k \in \N$.
 \item If $\beta_{v_1},\beta_{v_2} < 0$, or $\beta_{v_1}\cdot \beta_{v_2} < 0$ and $\beta_{v_0} \geq 0$, then $\lambda_k (H) \leq \lambda_k (\widetilde H)$ for all $k \in \N$.
\end{enumerate}
\end{theorem}

\begin{proof}
Consider again the quadratic forms $h$ and $\widetilde h$ corresponding to $H$ and $\widetilde H$, respectively. First, consider the case $\beta_{v_1},\beta_{v_2},\beta_{v_0}\neq 0$. Then $\dom h = \dom\widetilde{h}$, and for $f\in\dom h$ we have
\begin{align*}
 \widetilde{h}(f) - h(f) & = \frac{1}{\beta_{v_0}}\left|\sum_{j=1}^{\deg(v_0)} F_j (v_0)\right|^2 - \frac{1}{\beta_{v_1}}\left|\sum_{j=1}^{\deg(v_1)} F_j (v_1)\right|^2 - \frac{1}{\beta_{v_2}}\left|\sum_{j=1}^{\deg(v_2)} F_j (v_2)\right|^2.
\end{align*}
Now, a case distinction with respect to the signs of $\beta_{v_1},\beta_{v_2},\beta_{v_0}$ yield the corresponding statements.

For $\beta_{v_0} = 0$, we have $\beta_{v_1} = -\beta_{v_2}$ as well as $\dom \widetilde{h} \subseteq \dom h$ and $\sum_{j=1}^{\deg(v_0)} F_j(v_0) = 0$ for $f\in \dom\widetilde{h}$, and therefore
\[\sum_{j=1}^{\deg(v_1)} F_j (v_1) = - \sum_{j=1}^{\deg(v_2)} F_j (v_2).\]
Thus, $\widetilde{h}(f) - h(f) = 0$ for all $f\in \dom \widetilde{h}$ and the assertion follows.

In case $\beta_{v_1}=0$ or $\beta_{v_2}=0$ we have $\dom h \subseteq \dom \widetilde{h}$ and again $\widetilde{h}(f) - h(f) = 0$ for all $f\in \dom h$.
\end{proof}

\subsection{Unfolding parallel and pendant edges}%\marginpar{\tiny JR: Symmetrizing parallel edges habe ich doch nicht gemacht, da das Resultat in \cite{BKKM19} in dem Fall von Eigenschaften der Eigenfunktion abhängt.}

In this section we consider a transformation, called unfolding, which replaces several parallel edges or several pendant edges by one single edge. We say that edges $e_1, \dots, e_r \in \cE$ are {\em parallel} if all of them are incident to the same two vertices. That is, either there exist $v_1, v_2 \in \cV$ such that $v_1 \neq v_2$ and $\{e_1, \dots, e_r\} \subseteq \cE_{v_1} \cap \cE_{v_2}$, or each of the edges $e_1, \dots, e_r$ is a loop attached to the same vertex $v \in \cV$. 

\begin{definition}%\label{def:unfolding}
Let $\Gamma$ be a compact metric graph in which the edges $e_1, \dots, e_r$ are parallel. We say that the graph $\widetilde \Gamma$ is obtained from $\Gamma$ by {\em unfolding the parallel edges $e_1, \dots, e_r$} if $\widetilde \Gamma$ has the same set of vertices as $\Gamma$, the same set of edges and edge lengths as $\Gamma$ except for $e_1, \dots, e_r$ and one edge $\widehat e$ of length $L (e_1) + \dots + L (e_r)$ incident to the same vertices as each of the edges $e_1, \dots, e_r$; cf.\ Figure \ref{fig:unfolding_parallel}.
\end{definition}

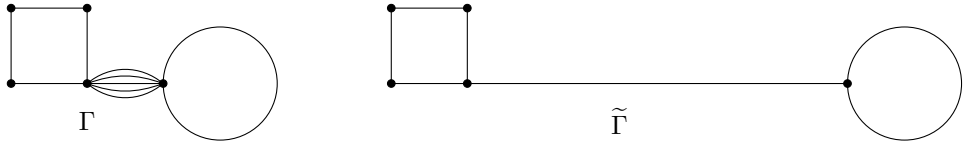
\begin{figure}[htb]
  \centering
  \begin{tikzpicture}
      \draw[fill] (0,0) circle(0.05);
      \draw[fill] (1,0) circle(0.05);
      \draw[fill] (1,1) circle(0.05);
      \draw[fill] (0,1) circle(0.05);
      \draw (0,0)--(1,0)--(1,1)--(0,1)--(0,0);
      \draw (1,-0.5) node[]{$\Gamma$};
      \draw[fill] (2,0) circle(0.05);
      \draw (1,0) to[out=0, in=180] (2,0);
      %\draw (1,0) to[out=10, in=170] (2,0);
      \draw (1,0) to[out=20, in=160] (2,0);
      \draw (1,0) to[out=40, in=140] (2,0);
      %\draw (1,0) to[out=-10, in=190] (2,0);
      \draw (1,0) to[out=-20, in=200] (2,0);
      \draw (1,0) to[out=-40, in=220] (2,0);
      \draw (2.75,0) circle(0.75);
      
    \begin{scope}[shift={(5,0)}]
      \draw[fill] (0,0) circle(0.05);
      \draw[fill] (1,0) circle(0.05);
      \draw[fill] (1,1) circle(0.05);
      \draw[fill] (0,1) circle(0.05);
      \draw (0,0)--(1,0)--(1,1)--(0,1)--(0,0);
      \draw (3,-0.5) node[]{$\widetilde\Gamma$};
      \draw[fill] (6,0) circle(0.05);
      \draw (1,0) to[out=0, in=180] (6,0);
      \draw (6.75,0) circle(0.75);
    \end{scope}
  \end{tikzpicture}
  \caption{Unfolding parallel edges in $\Gamma$.}
  \label{fig:unfolding_parallel}
\end{figure}

The following theorem will be formulated for $\delta$ vertex conditions. It was shown more generally, including the lowest non-trivial eigenvalue for continuity-Kirchhoff vertex conditions, in \cite[Theorem 3.18]{BKKM19}. In contrast to \cite{BKKM19}, we admit potentials on the edges and demonstrate that the proof from \cite{BKKM19} carries over to this case.

\begin{theorem}\label{thm:unfolding}
Let $\Gamma$ be a compact metric graph in which the edges $e_1, \dots, e_r$ are parallel. Let $H$ be a Schrödinger operator on $\Gamma$ with real potential $q \in L^\infty (\Gamma)$ on the edges such that 
\begin{align*}
 q |_{e_j} \geq 0, \quad j = 1, \dots, r,
\end{align*}
and with $\delta$ vertex conditions at all vertices. Furthermore, let $\widetilde \Gamma$ be the graph obtained by unfolding the parallel edges $e_1, \dots, e_r$, and let $\widetilde H$ be the Schrödinger operator on~$\widetilde \Gamma$ with the same potential as for $H$ on all edges except for the new edge $\widehat e$ and the constant zero-potential on $\widehat e$. Moreover, for $\widetilde H$ we assume that on all vertices $\delta$ coupling conditions with the same strengths as for $H$ are imposed. Then
\begin{align*}
 \lambda_1 (\widetilde H) \leq \lambda_1 (H)
\end{align*}
holds. If $\lambda_1 (H) \neq 0$, then even $\lambda_1 (\widetilde H) < \lambda_1 (H)$.
\end{theorem}

\begin{proof}
We will prove the theorem in the case $r = 2$, i.e.\ for unfolding two parallel edges $e_1, e_2$. The general case can be obtained by iterating the following procedure. Let $v_1$ and $v_2$ be the vertices to which $e_1$ and $e_2$ are incident. We consider only the case that $v_1$ and $v_2$ are distinct vertices; the case in which $v_1 = v_2$, i.e.\ $e_1$ and $e_2$ are loops, is completely analogous.

By Theorem \ref{thm:courant} we can choose $f \in \ker (H - \lambda_1 (H))$ strictly positive on $\Gamma$. Consider the restrictions $f_1$ and $f_2$ of $f$ to $e_1$ and $e_2$, respectively. Let, without loss of generality,
\begin{align*}
 a := f_1 (0) = f_2 (0) > 0, \qquad b := f_1 (L (e_1)) = f_2 (L (e_2)) > 0,
\end{align*}
and assume without loss of generality $a < b$. Moreover, choose $x_0 \in [0, L (e_1)]$ such that
\begin{align*}
 f_1 (x_0) = a \quad \text{and} \quad f_1 ([x_0, L (e_1)]) = [a, b].
\end{align*} 
We distinguish two cases.

{\bf Case 1:} $\lambda_1 (H) \geq 0$. In this case, the quadratic form $h$ corresponding to $H$ satisfies $h (f) \geq 0$. Define a function $\widetilde f$ on the graph $\widetilde H$ by letting it be equal to $f$ on all edges except $\widehat e$ (after identification of these edges with edges of $\Gamma$ in the natural way). On $\widehat e$, identified as $[0, L (e_1) + L (e_2)]$, we define $\widehat f$ by inserting the edge $e_2$ and the function $f_2$ on it into $e_1$ at the point $x_0$ and defining $\widetilde f$ to be equal to $f_1$ to the left of $x_0$ (if not $x_0 = 0$) and constantly equal to $b$ to the right. That is,
\begin{align*}
 \widetilde f_{\widehat e} (x) = \begin{cases}
                                  f_1 (x), & x \in [0, x_0],\\
                                  f_2 (x - x_0), & x \in [x_0, x_0 + L (e_2)], \\
                                  b, & x \in [x_0 + L (e_2), L (e_1) + L (e_2)].
                                 \end{cases}
\end{align*}
This construction is made such that the function $\widetilde f$ belongs to $\widetilde H^1 (\widetilde \Gamma)$ and satisfies 
\begin{align*}
 \widetilde f (v_1) = a = f (v_1) \quad \text{and} \quad \widetilde f (v_2) = b = f (v_2),
\end{align*}
in particular, it belongs to the domain of the quadratic form $\widetilde h$ corresponding to $\widetilde H$. Furthermore,
\begin{align*}
 \widetilde h (\widetilde f) & = h (f) - \int_{e_1 \cup e_2} \big(|f'|^2 + q |f|^2 \big) + \int_{\widehat e} |\widetilde f'|^2 \leq h (f),
\end{align*}
where we have used that $q$ is non-negative on $e_1 \cup e_2$ as well as
\begin{align}\label{eq:formUnfold}
 \int_{e_1 \cup e_2} |f'|^2 \geq \int_{\widehat e} |\widetilde f'|^2.
\end{align}
Moreover,
\begin{align}\label{eq:normUnfold}
 \int_{\widetilde \Gamma} |\widetilde f|^2 & = \int_\Gamma |f|^2 - \int_{e_1 \cup e_2} |f|^2 + \int_{\widehat e} |\widetilde f|^2 \geq \int_\Gamma |f|^2
\end{align}
since 
\begin{align*}
 \int_{\widehat e} |\widetilde f|^2 & = \int_0^{x_0} |f_1|^2 + \int_0^{L (e_2)} |f_2|^2 + b^2 (L (e_1) - x_0) \geq \int_{e_1} |f_1|^2 + \int_{e_2} |f_2|^2.
\end{align*}
From \eqref{eq:formUnfold}, \eqref{eq:normUnfold} and $h (f) \geq 0$ we obtain
\begin{align}\label{eq:naAlso!}
 \lambda_1 (\widetilde H) & \leq \frac{\widetilde h (\widetilde f)}{\int_{\widetilde \Gamma} |\widetilde f|^2} \leq \frac{h (f)}{\int_\Gamma |f|^2} = \lambda_1 (H).
\end{align}

{\bf Case 2:} $\lambda_1 (H) < 0$. In this case we do a little modification of the construction in Case~1 by defining
\begin{align*}
 \widetilde f_{\widehat e} (x) = \begin{cases}
                                  f_1 (x), & x \in [0, x_0],\\
                                  a, & x \in [x_0, L (e_1)], \\
                                  f_2 (x - L (e_1)), & x \in [L (e_1), L (e_1) + L (e_2)].
                                 \end{cases}
\end{align*}
Then computations analogous to those in Case~1 yield
\begin{align*}
 \widetilde h (\widetilde f) \leq h (f) < 0
\end{align*}
and
\begin{align*}
 \int_{\widetilde \Gamma} |\widetilde f|^2 \leq \int_\Gamma |f|^2.
\end{align*}
The last two relations lead again to \eqref{eq:naAlso!}.

If $\lambda_1 (H) \neq 0$, then $f_1$ is not constant and, hence, the inequality \eqref{eq:naAlso!}, respectively its counterpart in Case 2, is strict. This completes the proof.
\end{proof}

The argument carried out in the above proof relies crucially on the continuity of the eigenfunction at the vertices. The latter allows to construct a test function which is continuous on the whole unfolded edge $\widehat e$. For vertex conditions that do not require continuity at the vertices, the assertion of Theorem \ref{thm:unfolding} need not be true, as the next example shows.

\begin{example}\label{ex:unfoldingAK}
Consider a graph $\Gamma$ consisting of two vertices $v_1$, $v_2$ and two parallel edges of arbitrary finite lengths $\ell_1, \ell_2$ connecting them. We let $H$ be the Schrödinger operator on $\Gamma$ corresponding to $q = 0$ identically on the edges and anti-Kirchhoff vertex conditions at both $v_1$ and $v_2$. Then $\lambda_1 (H) = 0$, the corresponding eigenspace being spanned by the function that is constantly 1 on the first edge and $-1$ on the second one. On the other hand, after unfolding $\Gamma$, we obtain the graph $\widetilde \Gamma$ isomorphic to an interval of length $\ell_1 + \ell_2$, see Figure \ref{fig:unfolding_parallel_example},
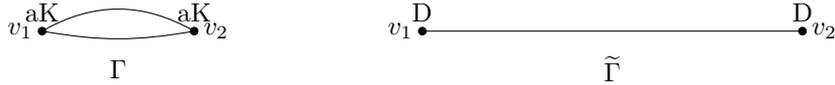
\begin{figure}[htb]
  \centering
  \begin{tikzpicture}
      \draw[fill] (0,0) circle(0.05) node[left]{$v_1$} node[above]{aK};
      \draw[fill] (2,0) circle(0.05) node[right]{$v_2$}node[above]{aK};
      \draw (0,0) to[out=-10, in=190] (2,0);
      \draw (0,0) to[out=30, in=150] (2,0);
      \draw (1,-0.5) node{$\Gamma$};
      
    \begin{scope}[shift={(5,0)}]
      \draw[fill] (0,0) circle(0.05)node[left]{$v_1$}node[above]{D};
      \draw[fill] (5,0) circle(0.05)node[right]{$v_2$}node[above]{D};
      \draw (0,0) to[out=0, in=180] (5,0);
      \draw (2.5,-0.5) node{$\widetilde{\Gamma}$};
    \end{scope}
  \end{tikzpicture}
  \caption{Unfolding parallel edges in the graph of Example \ref{ex:unfoldingAK}.}
  \label{fig:unfolding_parallel_example}
\end{figure}
and the Laplacian $\widetilde H$ on $\widetilde \Gamma$ subject to anti-Kirchhoff vertex conditions equals the Laplacian on the interval with Dirichlet boundary conditions at both end points. In particular,
\begin{align*}
 \lambda_1 (\widetilde H) > 0 = \lambda_1 (H).
\end{align*}

We would like to point out that the exact same construction works if the anti-Kirchhoff conditions in the above example are replaced by $\delta'$ vertex conditions with positive strengths.
\end{example}

\begin{openproblem}
Example \ref{ex:unfoldingAK} shows that for $\delta'$ vertex conditions in some cases unfolding parallel edges leads to an increase of the principal eigenvalue. Can one prove a general statement of this type for $\delta'$ vertex conditions, possibly depending on the signs of the coupling strengths?
\end{openproblem}

Next we consider a related surgical graph manipulation introduced in \cite{BKKM19}, called unfolding pendant edges; we call an edge $e$ {\em pendant} if it is incident to a vertex of degree one. 

\begin{definition}
Let $\Gamma$ be a compact metric graph in which the edges $e_1, \dots, e_r$ are pendant and all incident to the same vertex $v_0$; that is, $e_j$ is incident to $v_0$ and a vertex $v_j$ of degree one. We say that the graph $\widetilde \Gamma$ is obtained from $\Gamma$ by {\em unfolding the pendant edges $e_1, \dots, e_r$} if $\widetilde \Gamma$ has the same set of vertices as $\Gamma$ except for $v_1, \dots, v_r$ and the same set of edges except for $e_1, \dots, e_r$ and, in addition, an edge $\widehat e$ of length $L (e_1) + \dots + L (e_r)$ connecting $v_0$ to a new vertex $\widehat v$ of degree one; cf.\ Figure \ref{fig:unfolding_pendant}.
\end{definition}

\begin{figure}[htb]
  \centering
  \begin{tikzpicture}
      \draw[fill] (0,0) circle(0.05);
      \draw[fill] (1,0) circle(0.05) node[below]{$v_0$};
      \draw[fill] (1,1) circle(0.05);
      \draw[fill] (0,1) circle(0.05);
      \draw (0,0)--(1,0)--(1,1)--(0,1)--(0,0);
      \draw (1,-0.5) node[]{$\Gamma$};
      \draw[fill] (3,0) circle(0.05);
      \draw[fill] (2.41,1.41) circle(0.05);
      \draw[fill] (2.41,-1.41) circle(0.05);
      \draw (1,0)--(3,0);
      \draw (1,0)--(2.41,1.41);
      \draw (1,0)--(2.41,-1.41);
      
    \begin{scope}[shift={(5,0)}]
      \draw[fill] (0,0) circle(0.05);
      \draw[fill] (1,0) circle(0.05) node[below]{$v_0$};
      \draw[fill] (1,1) circle(0.05);
      \draw[fill] (0,1) circle(0.05);
      \draw (0,0)--(1,0)--(1,1)--(0,1)--(0,0);
      \draw (3,-0.5) node[]{$\widetilde\Gamma$};
      \draw[fill] (6,0) circle(0.05);
      \draw (1,0)--(6,0);
    \end{scope}
  \end{tikzpicture}
  \caption{Unfolding pendant edges in $\Gamma$.}
  \label{fig:unfolding_pendant}
\end{figure}
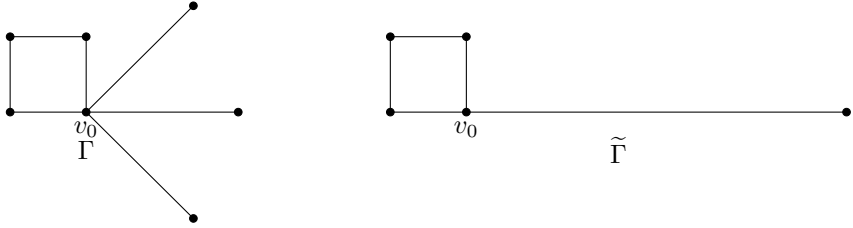

The following theorem is a variant of Theorem 3.18 (4) in \cite{BKKM19}, where $\delta$ or Dirichlet conditions are imposed at all vertices were considered and the case of the second eigenvalue for continuity-Kirchhoff conditions was included; on the other hand, the theorem was stated in the potential-free case in \cite{BKKM19}.

\begin{theorem}%\label{thm:unfoldingPendant}
Let $\Gamma$ be a compact metric graph in which the edges $e_1, \dots, e_r$ are pendant and incident to the same vertex $v_0$, and let $v_j$ denote the vertex distinct from $v_0$ to which $e_j$ is incident, $j = 1, \dots, r$. Let $H$ be a Schrödinger operator on $\Gamma$ with real potential $q \in L^\infty (\Gamma)$ on the edges such that 
\begin{align*}
 q |_{e_j} \geq 0, \quad j = 1, \dots, r,
\end{align*}
and a $\delta$ vertex condition with strength $\alpha_v$ at each vertex $v \in \cV$. We assume
\begin{align*}
 \alpha_{v_j} \geq 0, \quad j = 1, \dots, r.
\end{align*}
Furthermore, let $\widetilde \Gamma$ be the graph obtained by unfolding the pendant edges $e_1, \dots, e_r$ into one edge $\widehat e$ incident to a new vertex $\widehat v$ of degree one, and let $\widetilde H$ be the Schrödinger operator on $\widetilde \Gamma$ with the same potentials as for $H$ on all edges except the new edge $\widehat e$ and the constant zero-potential on $\widehat e$. Moreover, for $\widetilde H$ we assume that on all vertices except $\widehat v$ (but including $v_0$) $\delta$ vertex conditions with the same strengths as for $H$ are imposed, and $\widehat v$ is equipped with a Neumann boundary condition. Then
\begin{align*}
 \lambda_1 (\widetilde H) \leq \lambda_1 (H)
\end{align*}
holds. If $\lambda_1 (H) \neq 0$, then even $\lambda_1 (\widetilde H) < \lambda_1 (H)$ holds.
\end{theorem}

\begin{proof}
The proof is very similar to the proof of Theorem \ref{thm:unfolding} and is only sketched here. We may assume here that $r = 2$, i.e.\ only two pendant edges are unfolded. The general case can be obtained through multiple applications of this case. Starting with a uniformly positive function $f \in \ker (H - \lambda_1 (H))$, we consider a test function $\widetilde f$ on $\widetilde \Gamma$ which is equal to $f$ on all edges except $\widehat e$, employing the natural identification of those edges with edges in the original graph $\Gamma$. On the new edge $\widehat e$, we define $\widetilde f$ depending on whether $\lambda_1 (H)$ is non-negative or negative. If $\lambda_1 (H) \geq 0$, choose $x_0 \in e_1 \cup e_2$ such that $f |_{e_1 \cup e_2}$ takes its maximum at $x_0$; without loss of generality we assume $x_0 \in e_1$. Then, on $\widehat e$, identified with the interval $[0, L (e_1) + L (e_2)]$, we set
\begin{align*}
 \widetilde f_{\widehat e} (x) = \begin{cases}
                                  f_1 (x), & x \in [0, x_0],\\
                                  f_1 (x_0), & x \in [x_0, x_0 + L (e_2)], \\
                                  f_1 (x - L (e_2)), & x \in [x_0 + L (e_2), L (e_1) + L (e_2)];
                                 \end{cases}
\end{align*}
in other words, we transplant the edge $e_2$ into $e_1$ at the position $x_0$ and replace the original eigenfunction on $e_2$ by a suitable constant. This construction yields a function $\widetilde f$ on $\widetilde \Gamma$ which is continuous inside each edge and at the joint vertex $v_0$ and belongs to the domain of the quadratic form $\widetilde{h}$ associated with $\widetilde H$. Now a computation analogous to the one in the proof of Theorem \ref{thm:unfolding} leads to
\begin{align*}
 \lambda_1 (\widetilde H) & \leq \frac{\widetilde h (\widetilde f)}{\int_{\widetilde \Gamma} |\widetilde f|^2} \leq \frac{h (f)}{\int_\Gamma |f|^2} = \lambda_1 (H).
\end{align*}
The case where $\lambda_1 (H) \leq 0$ can be treated analogously; the only difference is that $x_0$ is chosen such that $f |_{e_1 \cup e_2}$ takes its minimum -- instead of maximum -- there.

If $\lambda_1 (H) \neq 0$ then $f_2$ is not constant and thus the above computations even yield strict inequality for the Rayleigh quotients. This implies the strict eigenvalue inequality.
\end{proof}

The construction in the previous proof makes use of the fact that removing pendant edges (for transplantation to a different position) does neither change the vertex condition required in the domain of the quadratic form at the vertex $v_0$ nor the vertex term in the quadratic form. This is fundamentally different for, e.g., $\delta'$ or anti-Kirchhoff vertex conditions. In fact, the following example shows that for such conditions the statement of the previous theorem may fail.

\begin{example}\label{ex:unfoldingPendAK}
Consider a star graph $\Gamma$ consisting of two pendant edges with lengths $\ell_1, \ell_2$ attached to the same vertex $v_0$ and the Laplacian on $\Gamma$ with an anti-Kirchhoff vertex condition at $v_0$ and Neumann vertex conditions at the degree-one vertices. Then $\lambda_1 (H) = 0$ and the eigenspace is spanned by the function which is constantly $1$ on one edge and $- 1$ on the other. However, unfolding the pendant edges and keeping the anti-Kirchhoff condition at $v_0$ results, effectively, in an interval with a Dirichlet boundary condition at one end and a Neumann boundary condition at the other end, see Figure \ref{fig:unfolding_pendant_example}. 
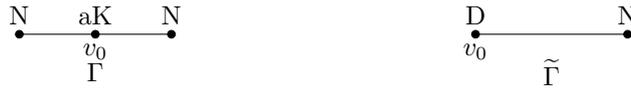
\begin{figure}[htb]
  \centering
  \begin{tikzpicture}
      \draw[fill] (-1,0) circle(0.05) node[above]{N};
      \draw[fill] (0,0) circle(0.05) node[below]{$v_0$} node[above]{aK};
      \draw[fill] (1,0) circle(0.05) node[above]{N};
      \draw (-1,0)--(0,0)--(1,0);
      \draw (0,-0.5) node[]{$\Gamma$};
      
    \begin{scope}[shift={(5,0)}]
      \draw[fill] (0,0) circle(0.05) node[below]{$v_0$}  node[above]{D};
      \draw[fill] (2,0) circle(0.05)  node[above]{N};
      \draw (0,0)--(2,0);
      \draw (1,-0.5) node[]{$\widetilde\Gamma$};
    \end{scope}
  \end{tikzpicture}
  \caption{Unfolding pendant edges in Example \ref{ex:unfoldingPendAK}.}
  \label{fig:unfolding_pendant_example}
\end{figure}
However, the Laplacian $\widetilde H$ subject to these conditions satisfies
\begin{align*}
 \lambda_1 (\widetilde H) = \frac{\pi^2}{4 (\ell_1 + \ell_2)^2} > 0 = \lambda_1 (H).
\end{align*}

The exact same construction works if the anti-Kirchhoff condition is replaced by a $\delta'$ condition with positive strength. In this case the same spectral effect occurs.
\end{example}

\begin{openproblem}
In Example \ref{ex:unfoldingPendAK} we have seen a situation where unfolding pendant edges for $\delta'$ vertex conditions increases the first eigenvalue. Can one show a general statement of this type, possibly depending on the signs of the strengths in the $\delta'$ couplings?
\end{openproblem}

\blue{

\section{Hadamard-type formulas}
\label{sec:Hadamard}

In this section we study Haramard-type formulas, i.e.\ the variation of eigenvalues and eigenvectors with respect to perturbations of the vertex conditions.

\begin{remark}
\label{rem:Frechet_derivative}
In general, the vertex conditions are parametrized by certain matrices, i.e.\ by elements of a Banach space, and differentiating has to be understood in the sense of, e.g., the Fr\'echet derivative: let $X$ be a Banach space, $U\subseteq X$ open, and $F\from U\to \C$ Fr\'echet differentiable; cf.\ \cite[Section~2.1]{Zeidler}. Then $F'$ is a mapping from $U$ to $\cL(X, \C)$, where $\cL(X,\C)$ denotes the space of bounded linear functionals on $X$. Thus, for $x\in U$, $F'(x)\from X\to \C$ is linear and bounded. In our case, $X$ will be finite-dimensional, $X = \C^{d\times d}$ for some $d$, and then for $x\in U$ we can represent $F'(x)$ by its matrix of partial derivatives.
\end{remark}

\begin{theorem}
\label{thm:Hadamard}
    Let $\Gamma$ be a compact metric graph, $q\in L^\infty(\Gamma)$ real, $v_0\in \cV$ and $H$ the Schr\"odinger operator on $\Gamma$ with potential $q$, arbitrary self-adjoint coupling conditions as in Proposition \ref{prop:SAconditions} at the vertices $v\in \cV\setminus\{v_0\}$ and a self-adjoint coupling condition with non-trivial Robin part, $P_{v_0, \rm R} \neq 0$, at the vertex $v_0$. Denote by $\Lambda_{v_0}$ the self-adjoint, invertible coupling operator in $\ran P_{v_0, \rm R}$.
    Let $\lambda$ be a simple eigenvalue of $H$ and $f$ be a corresponding normalized eigenfunction.
    Then $\lambda$ and $f$ are differentiable w.r.t.\ $\Lambda_{v_0}$ and we have
    \[\lambda'(\Lambda_{v_0})\from \widetilde{\Lambda_{v_0}} \mapsto \big\langle \widetilde{\Lambda_{v_0}} P_{v_0, \rm R} F (v_0), P_{v_0, \rm R} F (v_0) \big\rangle\]
    that is, 
    \[\lambda'(\Lambda_{v_0})(\widetilde{\Lambda_{v_0}}) = \big\langle \widetilde{\Lambda_{v_0}} P_{v_0, \rm R} F (v_0), P_{v_0, \rm R} F (v_0) \big\rangle\]
    for all operators $\widetilde{\Lambda_{v_0}}$ in $\ran P_{v_0,\rm R}$.
\end{theorem}

\begin{remark}
We would like to point out that $\ran P_{v_0, \rm R}$ is a finite-dimensional vector space, implying that $\Lambda_{v_0}$ can be represented as a Hermitian matrix. Hence, differentiating with respect to $\Lambda_{v_0}$ in fact is just differentiating on a finite-dimensional space.
\end{remark}

\begin{proof}[Proof of Theorem \ref{thm:Hadamard}]
    By \cite[Theorem 3.8 and Theorem 3.10]{BK12} and the different equivalent ways to describe vertex conditions (see e.g.\ \cite[Theorem 5]{kuc08}) the eigenvalues and eigenfunctions are differentiable with respect to $\Lambda_{v_0}$. Let $d:=\dim \ran P_{v_0,\rm R}>0$ and $j,k\in \{1,\ldots, d\}$. We abbreviate $\partial_{jk}:=\frac{\partial}{\partial (\Lambda_{v_0})_{j,k}}$ for the partial derivatives.
    
    We check that partial derivatives of $f$ w.r.t.\ components of the parameter again belong to $\dom h$. It suffices to check that the derivative is in $\widetilde{H}^1(\Gamma)$, as the condition $P_{v,\rm D}F(v) = 0$ for all $v\in \cV$ is a closed condition and thus stable under taking partial derivatives w.r.t.\ parameters.    
    Note that on each edge $e\in \cE$, $f$ is a linear combination of two basis functions solving the eigenvalue equation $-f'' + qf = \lambda f$; the dependence of $f$ on $\Lambda_{v_0}$ is only present in the coefficients of these basis functions. Thus, on each edge, $\partial_{jk} f$ is again in $H^1$ on each edge (as edges have finite length). Thus, $\partial_{jk} f\in \widetilde{H}^1(\Gamma)$, and hence $\partial_{jk} f\in \dom h$.
    
    Since $f$ is normalized, we obtain
    \[2\Real \langle f, \partial_{jk} f\rangle = 0, \quad j, k = 1, \dots, d.\]
    Moreover, as $h(f) = \lambda$, we get
    \[\partial_{jk} \lambda (\Lambda_{v_0})= \bigl(P_{v_0, \rm R} F (v_0)\bigr)_k \overline{\bigl(P_{v_0, \rm R} F (v_0)\bigr)_j} + 2\Real h(f,\partial_{jk}f).\]
    Now, since $f\in \dom H$ with $Hf=\lambda f$ and $\partial_{jk}f\in \dom h$, we observe
    \[\Real h(f,\partial_{jk}f) = \Real \langle Hf, \partial_{jk} f\rangle = \lambda \Real \langle f, \partial_{jk} f\rangle = 0,\]
    and therefore
    \[\partial_{jk} \lambda (\Lambda_{v_0})= \bigl(P_{v_0, \rm R} F (v_0)\bigr)_k \overline{\bigl(P_{v_0, \rm R} F (v_0)\bigr)_j}.\]
    We note that this partial derivative is actually constant.
    
    With respect to the canonical basis $\bigl((\delta_{jk,lm})_{l,m\in\{1,\ldots,d\}}: j,k\in \{1,\ldots,d\}\bigr)$ of $\C^{d\times d}$, we observe
    \begin{align*}
        \lambda'(\Lambda_{v_0}) ((\delta_{jk,lm})) & = \partial_{jk} \lambda(\Lambda_{v_0}) = \bigl(P_{v_0, \rm R} F (v_0)\bigr)_k \overline{\bigl(P_{v_0, \rm R} F (v_0)\bigr)_j} \\
        & = \langle ((\delta_{jk,lm}) P_{v_0, \rm R} F (v_0), P_{v_0, \rm R} F (v_0) \big\rangle.
    \end{align*}
    This yields the assertion.
\end{proof}

We will now specialize to $\delta$ and $\delta'$ coupling conditions; we will formulate this as two corollaries.

\begin{corollary}[{\cite[Proposition 4.2]{BK12}}]
\label{cor:Hadamard_delta}
Let $\Gamma$ be a compact metric graph, $q\in L^\infty(\Gamma)$ real, $v_0\in \cV$ and $H$ the Schr\"odinger operator on $\Gamma$ with potential $q$, arbitrary self-adjoint coupling conditions as in Proposition \ref{prop:SAconditions} at the vertices $v\in \cV\setminus\{v_0\}$ and a $\delta$ coupling condition of strength $\alpha$ at the vertex $v_0$.
    Let $\lambda$ be a simple eigenvalue and $f$ be a corresponding normalized eigenfunction.
    Then $\lambda$ and $f$ are differentiable w.r.t.\ $\alpha$ and we have
    \[\frac{\mathrm{d} \lambda}{\mathrm{d} \alpha}(\alpha) = |f(v_0)|^2.\]
\end{corollary}

\begin{proof}
    The proof follows from Theorem \ref{thm:Hadamard} and the fact that for the $\delta$ coupling condition we have that $\Lambda_{v_0}$ is multiplication by $\frac{\alpha}{\deg(v_0)}$.
\end{proof}

\begin{corollary}
\label{cor:Hadamard_delta'}
Let $\Gamma$ be a compact metric graph, $q\in L^\infty(\Gamma)$ real, $v_0\in \cV$ and $H$ the Schr\"odinger operator on $\Gamma$ with potential $q$, arbitrary self-adjoint coupling conditions as in Proposition \ref{prop:SAconditions} at the vertices $v\in \cV\setminus\{v_0\}$ and a $\delta'$ coupling condition of strength $\beta$ at the vertex $v_0$.
    Let $\lambda$ be a simple eigenvalue and $f$ be a corresponding normalized eigenfunction.
    Then $\lambda$ and $f$ are differentiable w.r.t.\ $\frac{1}{\beta}$ and w.r.t.\ $\beta$, and we have
    \[\frac{\mathrm{d} \lambda}{\mathrm{d} \frac{1}{\beta}}(\beta) = \bigg| \sum_{j=1}^{\deg(v_0)} F_j(v_0)\bigg| ^2 \qquad \text{and}\qquad \frac{\mathrm{d} \lambda}{\mathrm{d} \beta}(\beta) = -\frac{1}{\beta^2} \bigg| \sum_{j=1}^{\deg(v_0)} F_j(v_0)\bigg| ^2.\]
\end{corollary}

\begin{proof}
    The proof of the first equality follows from Theorem \ref{thm:Hadamard} and the fact that for the $\delta'$ coupling condition we have that $\Lambda_{v_0}$ is multiplication by $\frac{\deg{v_0}}{\beta}$. The derivative of $\lambda$ w.r.t.\ $\beta$ can then be obtained by the chain rule.
\end{proof}
}

\section{Bounds for the lowest eigenvalue}
\label{sec:bounds}

In this final section we review some known bounds on the ground state eigenvalue of a Schrödinger operator on a metric graph, for special choices of vertex conditions. We consciously exclude the case of continuity-Kirchhoff vertex conditions, where $\lambda_1 (H) = 0$, but would like to mention that a large body of literature dealing with estimates for the lowest positive eigenvalue (also called spectral gap) as well as its higher eigenvalues exists. We refer the reader to \cite{BL17,BKKM19,BCJ21,BHY23,F05,K20,KKMM16,K13,K15,KMN13,N87,P21,R17,R22} and the references therein.

For $\delta$ coupling conditions, one has the following lower bound, where the total positive, respectively negative interaction strengths of $H$ are defined as
\begin{align}\label{eq:totalStrengths}
 I_+ := \int_\Gamma q_+ + \sum_{v : \alpha_v > 0} \alpha_v \quad \text{and} \quad I_- := \int_\Gamma q_- - \sum_{v : \alpha_v < 0} \alpha_v,
\end{align}
denoting by $q_+, q_- \geq 0$ the positive and negative parts of the potential $q$.

\begin{theorem}[{\cite[Theorem 1]{KKT16}}]
Let $\Gamma$ be a compact, connected metric graph with total length $L$ and let $H$ be the Schrödinger operator on $\Gamma$ with real-valued potential $q \in L^{\infty} (\Gamma)$ and a $\delta$ coupling condition of strength $\alpha_v$ at each vertex $v \in \cV$. Moreover, let the total positive and negative interaction strengths of $H$ be given in \eqref{eq:totalStrengths}. Then the first eigenvalue $\lambda_1 (H)$ satisfies
\begin{align*}
 \lambda_1 (H) \geq \lambda_1 (\widehat H),
\end{align*}
where $\widehat H$ is the Schrödinger operator with potential zero (i.e.\ the Laplacian) on the interval of length $L$ with a $\delta$ vertex condition (i.e.\ Robin boundary condition) of strength $I_+$ at one end point and a $\delta$ vertex condition of strength $  I_-$ at the other end point. In other words, $\lambda_1 (H)$ is bounded from below by $k^2$, where $k$ the smallest solution to the secular equation
\begin{align*}
 \left( k + \frac{I_- I_+}{k} \right) \tan (k L) = I_+ - I_-;
\end{align*}
cf.\ Figure \ref{fig:interval}.
\end{theorem}

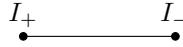
\begin{figure}[htb]
  \centering
  \begin{tikzpicture}
    \draw[fill] (0,0) circle(0.05) node[above]{$I_+$};
    \draw[fill] (2,0) circle(0.05) node[above]{$I_-$};
    \draw (0,0)--(2,0);
  \end{tikzpicture}
  \caption{The graph and $\delta$ strengths for $\widehat H$.}
  \label{fig:interval}
\end{figure}

The latter theorem may be viewed as an isoperimetric result: among the Schrö\-dinger operators with $\delta$ coupling conditions on all graphs of fixed length and fixed total positive and negative interaction strengths, the first eigenvalue gets minimal on an interval without edge potential.

The following example indicates that an upper bound for $\lambda_1 (H)$, in the case of $\delta$ vertex conditions, depending only on the strengths of the $\delta$ couplings and the total length of the graph might exist. For the first non-trivial eigenvalue of the Laplacian with continuity-Kirchhoff
vertex conditions, such a bound cannot exist, as simple counterexamples show; see, e.g., \cite{KKMM16}. However, the lowest eigenvalue is always zero in this case, hence bounded.

\begin{example}
Consider a star graph of total length $L = 1$ with $E$ edges of length $l = 1/E$, see Figure \ref{fig:n-star}. 
\begin{figure}[htb]
  \centering
  \begin{tikzpicture}
    \draw[fill] (0,0) circle(0.05);
    \foreach \ph in {0,45,90,135,180,225,270,315}{
        \draw[fill] ({2*cos(\ph)},{2*sin(\ph)}) circle(0.05);
        \draw (0,0)--({2*cos(\ph)},{2*sin(\ph)}) ;
    }
  \end{tikzpicture}
  \caption{The equilateral $E$-star graph for $E=8$.}
  \label{fig:n-star}
\end{figure}
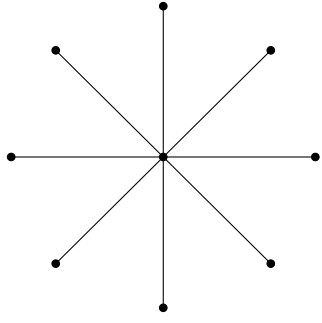
We impose Neumann boundary conditions at all vertices of degree one and a $\delta$ coupling condition of strength $\alpha = 1$ at the star vertex. If we parametrize the edges by $E$ copies of the interval $(0, l)$, with the end point $l$ corresponding to the star vertex, then the restriction $f_e$ of $f$ to any edge $e \in \cE$ is given by
\begin{align*}
 f_e (x) = A_e \cos (k x), \quad x \in (0, l),
\end{align*}
where $\lambda = k^2$ is the eigenvalue. Moreover, the vertex conditions at the star vertex can be phrased
\begin{align*}
 f_e (l)~\text{is independent of}~e \quad \text{and} \quad - \sum_{e \in \cE} f_e' (l) = f_e (l), \quad e \in \cE.
\end{align*}
By a simple computation, this gives rise to two cases, either
\begin{align*}
 k \in \frac{\pi}{2 l} \N
\end{align*}
or $k$ solves
\begin{align*}
 \frac{1}{k} = E \tan \left( \frac{k}{E} \right).
\end{align*}
As the right-hand side converges to $k$ as $E \to \infty$, there is a sequence of solutions $k_{E}$ of the above problem converging to $1$ as $E \to \infty$. Hence the lowest eigenvalue remains bounded. 
\end{example}

A trivial upper bound for the lowest eigenvalue of the Laplacian with $\delta$ coupling conditions may be obtained by inserting any constant function into the Rayleigh quotient:

\begin{proposition}\label{prop:deltaUpper}
Let $H$ be a Schr\"odinger operator on a compact metric graph $\Gamma$ with real-valued potential $q$ and a $\delta$ coupling condition with strength $\alpha_v$ at each vertex $v \in \cV$. Then
\begin{align*}
 \lambda_1 (H) \leq \frac{1}{L (\Gamma)} \left(\int_\Gamma q  + \sum_{v \in \cV} \alpha_v \right),
\end{align*}
where $L (\Gamma)$ denotes the total length of $\Gamma$.
\end{proposition}

Equality in the above estimate holds only if $q = 0$ and $\alpha_v = 0$ for all $v \in \cV$, the case in which $\lambda_1 (H) = 0$. However, in general the estimate may be rather rough, especially in the presence of large coupling coefficients: if $\alpha_v \to + \infty$ for all $v$, then $\lambda_1 (H)$ converges to the lowest eigenvalue of the Schrödinger operator with Dirichlet boundary conditions at each vertex -- in the potential-free case this is
\begin{align*}
 \frac{\pi^2}{L_{\max}^2} \leq \frac{E^2 \pi^2}{L (\Gamma)^2},
\end{align*}
where $L_{\max}$ is the largest edge length in $\Gamma$. However, the bound in Proposition \ref{prop:deltaUpper} tends to $+ \infty$ in this case.

The following proposition indicates that a careful study of the eigenvalues of flower graphs may give rise to better bounds.

\begin{figure}[htb]
  \centering
  \begin{tikzpicture}
    \draw[fill] (0,0) circle(0.05);
    \foreach \ph in {0,225}{
        \begin{scope}[rotate=\ph]
        \draw (0,0) to[out={-15}, in={180+90}] (2,0);
        \draw (0,0) to[out={+15}, in={180-90}] (2,0);
        \end{scope}
    }
    \foreach \ph in {90,180}{
        \begin{scope}[rotate=\ph]
        \draw (0,0) to[out={-15}, in={180+90}] (1.25,0);
        \draw (0,0) to[out={+15}, in={180-90}] (1.25,0);
        \end{scope}
    }
    \foreach \ph in {45,270}{
        \begin{scope}[rotate=\ph]
        \draw (0,0) to[out={-15}, in={180+90}] (1.5,0);
        \draw (0,0) to[out={+15}, in={180-90}] (1.5,0);
        \end{scope}
    }
    \foreach \ph in {135,315}{
        \begin{scope}[rotate=\ph]
        \draw (0,0) to[out={-15}, in={180+90}] (1.75,0);
        \draw (0,0) to[out={+15}, in={180-90}] (1.75,0);
        \end{scope}
    }
  \end{tikzpicture}
  \caption{A flower graph with $E=8$ edges.}
  \label{fig:Flower}
\end{figure}
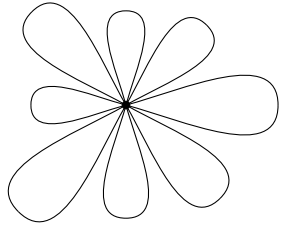

\begin{proposition}
Let $\Gamma$ be a compact metric graph with vertex set $\cV$, edge set $\cE$, and edge lengths $L (e)$, $e \in \cE$. Let $H$ be the Schrödinger operator on $\Gamma$ with real-valued potential $q$ and a $\delta$ coupling condition with strength $\alpha_v$ at each vertex $v \in \cV$. We assume that
\begin{align*}
 \alpha := \sum_{v \in \cV} \alpha_v > 0.
\end{align*}
Moreover, let $\widehat \Gamma$ be the flower graph on $E$ edges, see Figure \ref{fig:Flower}, with the same lengths as for $\Gamma$, and let $\widehat H$ be the Schrödinger operator on $\Gamma$ with the same potential $q$, transplanted to $\widehat \Gamma$ in the natural way, and a $\delta$ vertex condition with strength $\alpha$ at the only vertex. Then
\begin{align}\label{eq:estFlower}
 \lambda_1 (H) \leq \lambda_1 (\widehat H).
\end{align}
In particular, if $q = 0$ constantly on $\Gamma$, then $\lambda_1 (H)$ is bounded from above by $\lambda = k^2$, where $k$ is the smallest non-negative solution of the secular equation
\begin{align}\label{eq:secularFlower}
 2 k \sum_{e \in \cE} \tan \left( \frac{k L (e)}{2} \right) = \alpha.
\end{align}
\end{proposition}

\begin{proof}
The graph $\widehat \Gamma$ and the operator $\widehat H$ can be obtained from $\Gamma$ and $H$ by joining all vertices, and Theorem \ref{thm:joining} yields \eqref{eq:estFlower}. It remains to show that $\lambda_1 (\widehat H)$ corresponds to the lowest non-negative solution of \eqref{eq:secularFlower} in case $q = 0$ identically. Indeed, as $\lambda_1 (\widehat H) > 0$, the corresponding eigenfunction is a sin-type function, and it can be chosen positive on all of $\widehat \Gamma$ by Theorem \ref{thm:courant}. Moreover, its restriction $f_e$ to any edge $e$, parametrized as $[- L(e)/2, L(e)/2]$ has the same value at the end point $  L (e)/2$ as at $L (e)/2$. Therefore, 
\begin{align*}
 f_e (x) = A_e \cos (k x), \quad x \in \left[ -\frac{L (e)}{2}, \frac{L (e)}{2} \right],
\end{align*}
for each edge $e$. Plugging this into the $\delta$ vertex condition at the only vertex of $\widehat \Gamma$ yields \eqref{eq:secularFlower}.
\end{proof}

An analogous reasoning can be done for $\delta'$ vertex conditions if all coupling coefficients are negative, using Theorem \ref{thm:joining_delta'}~(ii). This results in the following statement.

\begin{proposition}
Let $\Gamma$ be a compact metric graph with vertex set $\cV$, edge set $\cE$, and edge lengths $L (e)$, $e \in \cE$. Let $H$ be the Schrödinger operator on $\Gamma$ with real-valued potential $q$ and a $\delta'$ coupling condition with strength $\beta_v < 0$ at each vertex $v \in \cV$. Let
\begin{align*}
 \beta := \sum_{v \in \cV} \beta_v.
\end{align*}
Moreover, let $\widehat \Gamma$ be the flower graph on $E$ edges with the same lengths as for $\Gamma$, and let $\widehat H$ be the Schrödinger operator on $\Gamma$ with the same potential $q$, transplanted to $\widehat \Gamma$ in the natural way, and a $\delta'$ vertex condition with strength $\beta$ at the only vertex. Then
\begin{align*}
 \lambda_1 (H) \leq \lambda_1 (\widehat H).
\end{align*}
\end{proposition}

We conclude this section by formulating two open problems.

\begin{openproblem}
Prove an upper bound for the principal eigenvalue of a Schrödinger operator with $\delta$ vertex couplings which is sharp, e.g.\ in an asymptotic sense when the coupling coefficients tend to $\infty$.
\end{openproblem}

Note that \cite[Theorem 1.7]{BSS2022} as well as \cite[Proposition 7]{BK2022} provide upper bounds for the difference of the eigenvalues with $\delta$ coupling and continuity-Kirchhoff conditions. This difference grows linearly in the $\delta$ vertex coupling in case of~\cite[Theorem 1.7]{BSS2022}.

\begin{openproblem}
Prove sharp upper and lower bounds for the lowest eigenvalue of a Schrödinger operator with $\delta'$ vertex conditions depending on the total length of the graph and the coupling strengths.
\end{openproblem}

% \section{Open problems and conjectures}
% \label{sec:conjectures}
% 
% 
% \todo{Machen}
% 
% \begin{itemize}
%  \item Upper bounds for ground state EV of $\delta$ depending on total length and total coupling strengths
%  \item Upper and lower bounds of that type for $\delta'$
% \end{itemize}

\end{document}